\documentclass{amsart}
\usepackage{amsfonts,amssymb}
\usepackage{multicol}

\newtheorem{theorem}{Theorem}
\newtheorem{lemma}{Lemma}
\newtheorem{corollary}{Corollary}

\newcommand{\integers}{{\mathbb Z}}

\def\ov{\overline}

\def\Nu{{\rm N}}
\def\Mu{{\rm M}}
\def\Tau{{\rm T}}
\def\Eta{{\rm H}}
\def\Kappa{{\rm K}}
\def\Iota{{\rm I}}
\def\circle{{\rm O}}

\begin{document}

\title{Fibered Orbifolds and Crystallographic Groups}

\author{John G. Ratcliffe and Steven T. Tschantz}

\address{Department of Mathematics, Vanderbilt University, Nashville, TN 37240
\vspace{.1in}}

\email{j.g.ratcliffe@vanderbilt.edu}


\date{}


\begin{abstract}
In this paper, we prove that 
a normal subgroup ${\rm N}$ of an $n$-dimen\-sional crystallographic 
group $\Gamma$ determines a geometric  
fibered orbifold structure on the flat orbifold $E^n/\Gamma$,  
and conversely every geometric fibered orbifold structure on $E^n/\Gamma$ 
is determined by a normal subgroup ${\rm N}$ of $\Gamma$. 
In particular, we prove that $E^n/\Gamma$ is a fiber bundle,  
with totally geodesic fibers, over a $\beta_1$-dimensional torus,  
where $\beta_1$ is the first Betti number of $\Gamma$. 

Let $\Nu$ be a normal subgroup of $\Gamma$ which is maximal 
in its commensurability class. 
We study the relationship between the exact sequence 
$1\to \Nu \to \Gamma\to \Gamma/\Nu\to 1$ splitting and the corresponding 
fibration projection having an affine section. 
If $\Nu$ is torsion-free, we prove that the exact sequence splits 
if and only if the fibration projection has an affine section. 
If the generic fiber $F = {\rm Span}(\Nu)/\Nu$ has an ordinary point 
that is fixed by every isometry of $F$, 
we prove that the exact sequence always splits. 
Finally, we describe all the geometric fibrations of the orbit spaces 
of all 2- and 3-dimensional crystallographic groups.  
\end{abstract}

\maketitle

\section{Introduction} 
Let $E^n$ be Euclidean $n$-space. 
A map $\phi:E^n\to E^n$ is an isometry of $E^n$ 
if and only if there is an $a\in E^n$ and an $A\in {\rm O}(n)$ such that 
$\phi(x) = a + Ax$ for each $x$ in $E^n$. 
We shall write $\phi = a+ A$. 
In particular, every translation $\tau = a + I$ is an isometry of $E^n$. 

A {\it flat $n$-orbifold} is a $(E^n,{\rm Isom}(E^n))$-orbifold 
as defined in \S 13.2 of Ratcliffe \cite{R}. 
A connected flat $n$-orbifold has a natural inner metric space structure. 
If $\Gamma$ is a discrete group of isometries of $E^n$, 
then its orbit space $E^n/\Gamma = \{\Gamma x: x\in E^n\}$ 
is a  connected, complete, flat $n$-orbifold, 
and conversely if $M$ is a connected, complete, flat $n$-orbifold, 
then there is a discrete group $\Gamma$ of isometries of $E^n$  
such that $M$ is isometric to $E^n/\Gamma$ by Theorem 13.3.10 of \cite{R}. 

\vspace{.15in}
\noindent{\bf Definition:}
A flat $n$-orbifold $M$ {\it geometrically fibers} over a flat $m$-orbifold $B$, 
with {\it generic fiber} a flat $(n-m)$-orbifold $F$, if there is a surjective map $\eta: M \to B$, 
called the {\it fibration projection},  
such that for each point $y$ of $B$,  
there is an open metric ball $B(y,r)$ of radius $r > 0$ centered at $y$ in $B$
such that $\eta$ is isometrically equivalent on $\eta^{-1}(B(y,r))$
to the natural projection $(F\times B_y)/G_y \to B_y/G_y$, 
where $G_y$ is a finite group acting diagonally on $F\times B_y$, isometrically on $F$,  
and effectively and orthogonally on an open metric ball $B_y$ in $E^m$ of radius $r$. 
This implies that the fiber $\eta^{-1}(y)$ is isometric to $F/G_y$. 
The fiber $\eta^{-1}(y)$ is said to be {\it generic} if $G_y = \{1\}$  
or {\it singular} if $G_y$ is nontrivial. 

\vspace{.15in}

An $n$-dimensional {\it crystallographic group} ({\it space group}) is a discrete group 
of isometries $\Gamma$ of $E^n$ such that $E^n/\Gamma$ is compact. 
We prove that if ${\rm N}$ is a normal subgroup of an $n$-dimensional space  
group $\Gamma$, then the flat orbifold $E^n/\Gamma$ geometrically fibers over 
a flat orbifold, with generic fiber a connected flat orbifold, 
naturally induced by ${\rm N}$. 
Conversely, we prove that if $E^n/\Gamma$ geometrically fibers over 
a flat orbifold $B$ with generic fiber a connected flat orbifold $F$, 
then this fibration is equivalent to a geometric fibration 
induced by a normal subgroup ${\rm N}$ of $\Gamma$. 

An {\it $m$-dimensional torus} or $m$-{\it torus}    
is a topological space homeomorphic to the cartesian product of $S^1$ 
with itself $m$-times. 
Here a 0-{\it torus} is defined to be a point. 
We prove that if $\Gamma$ is an $n$-dimensional space group $\Gamma$ 
with first Betti number $\beta_1$, then the flat orbifold $E^n/\Gamma$ 
is a fiber bundle over a $\beta_1$-torus with totally geodesic fibers. 

We illustrate the theory by describing all the geometric fibrations 
of the orbit spaces of all $2$- and $3$-dimensional space groups 
building on the work of Conway et al. \cite{C-T}.

\section{Normal Subgroups of Crystallographic Groups} 

The fundamental theorem of crystallographic groups is the following theorem. 

\begin{theorem} 
{\rm (Bieberbach's Theorems)}
\begin{enumerate}
\item If $\Gamma$ is an $n$-dimensional space group, 
then the subgroup ${\rm T}$ of all translations in $\Gamma$ is a free abelian 
normal subgroup of rank $n$ and of finite index in $\Gamma$ 
such that $\{a: a+I\in {\rm T}\}$ spans $E^n$. 
\item Two $n$-dimensional space groups $\Gamma_1$ and $\Gamma_2$ 
are isomorphic if and only if they are conjugate by an affine homeomorphism of $E^n$. 
\item There are only finitely many isomorphism classes of $n$-dimensional 
space groups for each $n$. 
\end{enumerate}
\end{theorem}

In dimensions $0,1,\ldots,6$, there are $1,2,17,219, 4\,783, 222\,018, 28\,927\,922$ 
isomorphism classes of space groups, respectively. 
See Brown et al. \cite{B-Z} and Plesken and Schulz \cite{P-S}.
 
Let $\Gamma$ be an $n$-dimensional space group. 
Define $\eta:\Gamma \to {\rm O}(n)$ by $\eta(a+A) = A$. 
Then $\eta$ is a homomorphism whose kernel is the group 
of translations in $\Gamma$. 
The image $\Pi$ of $\eta$ is a finite group by Part (1) of Theorem 1 called the 
{\it point group} of $\Gamma$. 
Let $\Nu$ be a normal subgroup of $\Gamma$. Define 
$${\rm Span}(\Nu) = {\rm Span}\{a\in E^n:a+I\in {\rm N}\}.$$

The following theorem strengthens Theorem 17 of Farkas \cite{Farkas}. 

\begin{theorem} 
Let ${\rm N}$ be a normal subgroup of an $n$-dimensional space group $\Gamma$, 
and let $V = {\rm Span}(\Nu)$. 
\begin{enumerate}
\item If $b+B\in\Gamma$, then $BV=V$. 
\item If $a+A\in \Nu$,  then $a\in V$ and $ V^\perp\subseteq{\rm Fix}(A)$. 
\item The group $\Nu$ acts effectively on each coset $V+x$ of $V$ in $E^n$ 
as a space group of isometries of $V+x$. 
\end{enumerate}
\end{theorem}
\begin{proof}
(1) Let $a+I\in {\rm N}$ and let $b+B\in\Gamma$, 
then $(b+B)(a+I)(b+B)^{-1} = Ba+I\in {\rm N}$. 
Hence $B$ leaves $V$ invariant. 

(2) The coset space $E^n/V$ is a Euclidean space 
where the distance between cosets is the orthogonal distance in $E^n$. 
The quotient map from $E^n$ to $E^n/V$ maps $V^\perp$ isometrically onto $E^n/V$. 
The group $\Gamma$ acts isometrically on $E^n/V$ by $(b+B)(V+x) = V+b+Bx$. 

Let ${\rm T} = \{a+I\in\Gamma\}$. 
Then ${\rm N}/{\rm N}\cap{\rm T} \cong {\rm NT}/{\rm T} \subseteq \Gamma/{\rm T}$, 
and so ${\rm N}/{\rm N}\cap{\rm T}$ is a finite group. 
The group ${\rm N}\cap{\rm T}$ acts trivially on $E^n/V$, 
and so ${\rm N}/{\rm N}\cap{\rm T}$ acts isometrically on $E^n/V$.  
Therefore ${\rm N}/{\rm N}\cap{\rm T}$ fixes a point $V+x$ of $E^n/V$. 
This implies that the group ${\rm N}$ leaves the coset $V+x$ invariant. 
Let ${\rm N}' = (-x+I){\rm N}(x+I)$. 
Then ${\rm N}'$ leaves $V$ invariant. 
Let $a+ A \in N$. 
Then $(-x+I)(a+A)(x+I) = a+ Ax-x+A$. 
Let $a' = a+ (A-I)x$. 
Then $a'+ A\in {\rm N}'$. 
As $a'+A$ leaves $V$ invariant, $a'\in V$. 

Let $b+I\in \Gamma$.  Then $b+I \in (-x+I)\Gamma(x+I)$, 
since $b+I$ and $x+I$ commute. 
Let $a+A\in {\rm N}$. 
Then $(-b+I)(a'+A)(b+I) = a'+(A-I)b+A$ is in ${\rm N}'$. 
Hence $(A-I)b\in V$. 
Now $\{b: b+I\in \Gamma\}$ spans $E^n$ by Theorem 1. 
Let $W = \big({\rm Fix}(A)\big)^\perp$. 
Then $A-I$ maps $W$ isomorphically onto $W$, 
and so $A-I$ maps $E^n$ onto $W$. 
Hence $\{(A-I)b:b+I\in\Gamma\}$ spans $W$. 
Therefore $W\subseteq V$. 
Hence $V^\perp \subseteq W^\perp = {\rm Fix}(A)$. 

Let $a+ A\in {\rm N}$. Write $a= b+c$ with $b\in V$ and $c\in V^\perp$. 
Let $r$ be the order of $A$. 
Then $r$ is finite and 
\begin{eqnarray*}
(a+A)^r  & = & a + Aa + \cdots + A^{r-1}a + I\\ 
& = & b + Ab+ \cdots + A^{r-1}b + rc +I\ \ = \ \ d + I. 
\end{eqnarray*}
As $b + Ab+ \cdots + A^{r-1}b,\ d \in V$, we have $rc=0$, and so $c=0$. 
Hence $a\in V$. 

(3) If $a+A\in\Nu$, then $a\in V$ and $AV=V$, and so $(a+A)V = V$. 
Hence the group  ${\rm N}$ leaves $V$ invariant. 
As $V^\perp \subseteq {\rm Fix}(A)$ for each $a+A\in {\rm N}$, 
we deduce that ${\rm N}$ acts effectively on $V+x$ for each $x\in V^\perp$. 

Let $a_1+I,\ldots, a_k+I$ be translations in ${\rm N}$ 
such that $\{a_1,\ldots,a_k\}$ is a basis for $V$. 
Then $(V+x)/\langle a_1+I,\ldots,a_k+I\rangle$ is a $k$-torus 
for each $x\in V^\perp$. 
Hence $(V+x)/{\rm N}$ is compact for each $x\in V^\perp$. 
Therefore ${\rm N}$ acts effectively as a space group of isometries of $V+x$ 
for each $x\in V^\perp$. 
\end{proof}

Let ${\rm N}$ be a normal subgroup of an $n$-dimensional space group $\Gamma$, 
and let $V = {\rm Span}(\Nu)$. 
The group ${\rm N}$ acts trivially on $E^n/V$ by Theorem 2, 
and so $\Gamma/{\rm N}$ acts isometrically on $E^n/V$ by
$$({\rm N}(b+B))(V+x) = {\rm N}((b+B)(V+x))= V+ b+Bx.$$ 
Observe that $\Nu(b+B)$ fixes $V$ if and only if $b\in V$. 
Hence the kernel of the corresponding homomorphism from $\Gamma/{\rm N}$ to ${\rm Isom}(E^n/V)$ 
is $\ov{\rm N}/{\rm N}$ where 
$$\ov{\rm N} = \{b+B\in \Gamma: b\in V\ \hbox{and}\ V^\perp\subseteq{\rm Fix}(B)\}.$$
As ${\rm N}\subseteq \ov{\rm N}$, the orbit space $V/{\rm N}$ projects onto $V/\ov{\rm N}$. 
The orbit space $V/{\rm N}$ is compact by Theorem 2,  
and so $V/\ov{\rm N}$ is compact. 
Hence, the group $\ov{\rm N}$ acts effectively 
as a space group of isometries of $V$. 
Therefore ${\rm N}$ has finite index in $\ov{\rm N}$. 

The group ${\rm N}$ may be a proper subgroup of $\ov{\rm N}$. 
For example, if ${\rm N} = \{a+I\in \Gamma\}$, then $\ov{\rm N} = \Gamma$. 
As $\ov{\rm N}/{\rm N}$ is a normal subgroup of $\Gamma/{\rm N}$, 
we have that $\ov{\rm N}$ is a normal subgroup of $\Gamma$. 
The group $\Gamma/\ov{\rm N}$ acts effectively on $E^n/V$ 
as a group of isometries. 

Let ${\rm H_1}$ and ${\rm H_2}$ be subgroups of a group $\Gamma$. 
Then ${\rm H_1}$ and ${\rm H_2}$ are said to be ${\it commensurable}$ 
if ${\rm H_1\cap H_2}$ has finite index in both ${\rm H_1}$ and ${\rm H_2}$. 
Commensurability is an equivalence relation on the set of subgroups of $\Gamma$. 

\begin{theorem} 
Let ${\rm N_1}$ and ${\rm N_2}$ be normal subgroups of a space group $\Gamma$. 
Then $\ov{\rm N_1} = \ov{\rm N_2}$ if and only if ${\rm N}_1$ and ${\rm N}_2$ 
are commensurable. 
\end{theorem}
\begin{proof}
Suppose $\Nu_1$ and $\Nu_2$ are commensurable. Let $\Tau=\{a+I\in\Gamma\}$. 
Now $\Nu_1\cap\Nu_2\cap\Tau$ has finite index in $\Nu_1\cap\Nu_2$ 
and $\Nu_1\cap\Nu_2$ has finite index in $\Nu_i$ for $i=1,2$. 
Hence $\Nu_1\cap\Nu_2\cap\Tau$ has finite index in $\Nu_i\cap\Tau$ for $i=1,2$. 
Therefore 
${\rm Span}(\Nu_1\cap\Nu_2) = {\rm Span}(\Nu_i)$
for $i = 1,2$. Hence $\ov{\rm N_1} = \ov{\rm N_2}$. 

Conversely, suppose $\ov{\rm N_1} = \ov{\rm N_2}$. 
Now $\Nu_i$ has finite index in $\ov{\Nu_i}$ for $i=1,2$. 
As $\Nu_1\Nu_2\subseteq \ov{\Nu_1}=\ov{\Nu_2}$, 
we have that $\Nu_i$ has finite index in $\Nu_1\Nu_2$ for $i=1,2$. 
Now $\Nu_1/\Nu_1\cap\Nu_2\cong \Nu_1\Nu_2/\Nu_2$ 
and $\Nu_2/\Nu_1\cap\Nu_2\cong \Nu_1\Nu_2/\Nu_1$. 
Hence $\Nu_1\cap\Nu_2$ has finite index in both $\Nu_1$ and $\Nu_2$. 
Therefore $\Nu_1$ and $\Nu_2$ are commensurable.  
\end{proof}

\begin{corollary} 
If $\Nu$ is a normal subgroup of a space group $\Gamma$, 
then $\ov{\ov\Nu} = \ov\Nu$ and 
$\ov\Nu$ is the unique maximal element of the commensurability class 
of normal subgroups of $\Gamma$ that contains $\Nu$. 
\end{corollary}

\begin{corollary} 
If $\psi: \Gamma \to \Gamma'$ is an isomorphism of space groups, 
and $\Nu$ is a normal subgroup of $\Gamma$, then $\ov{\psi(\Nu)} = \psi(\ov\Nu)$. 
\end{corollary}

\section{Geometric Flat Orbifold Fibrations} 

We say that the normal subgroup $\Nu$ of a space group 
$\Gamma$ is a {\it complete} if $\ov\Nu=\Nu$. 
If $\Nu$ is a normal subgroup of $\Gamma$, then $\ov\Nu$ 
is a complete normal subgroup of $\Gamma$ by Theorem 3 
called the {\it completion} of $\Nu$ in $\Gamma$. 

\begin{lemma} 
Let $\Nu$ be a complete normal subgroup of an $n$-dimensional space 
group $\Gamma$, and let $V={\rm Span}(\Nu)$. 
Then $\Gamma/\Nu$ acts effectively as a discrete group of isometries of $E^n/V$. 
\end{lemma}
\begin{proof}
From the above discussion, we know that $\Gamma/\Nu$ 
acts effectively as a group of isometries of $E^n/V$. 
As $\Nu\Tau/\Nu$ is of finite index in $\Gamma/\Nu$, 
it suffices to show that $\Nu\Tau/\Nu$ acts as a discrete group 
of isometries of $E^n/V$. 
Now $\Nu\Tau/\Nu\cong \Tau/\Nu\cap\Tau$. 
We claim that the group $\Tau/\Nu\cap\Tau$ is torsion-free. 
Let $\tau =a+I\in{\rm T}$ and suppose tha $\tau^r\in {\rm N}\cap{\rm T}$ for some integer $r>0$. 
Then $ra+I \in {\rm N}\cap{\rm T}$. Hence $ra\in V$, and so $a\in V$. 
Therefore $\tau\in {\rm N}\cap{\rm T}$, since $\Nu$ is complete.  
Thus ${\rm T}/{\rm N}\cap{\rm T}$ is torsion-free. 
Hence ${\rm T}$ has a set of generators $\{b_1+I,\ldots, b_n+I\}$ such that 
$b_1+I,\ldots, b_k+I$ generate ${\rm N}\cap{\rm T}$. 
Then ${\rm N}(b_{k+1}+I),\ldots,{\rm N}(b_n+I)$ generate ${\rm N}{\rm T}/{\rm N}$. 
As $\{b_{k+1},\ldots, b_n\}$ projects to a basis of $E^n/V$, we have that 
${\rm N}{\rm T}/{\rm N}$ acts as a discrete group of isometries of $E^n/V$ 
by Theorems 5.2.4 and 5.3.2 of Ratcliffe \cite{R}. 
\end{proof}

\begin{theorem}  
Let ${\rm N}$ be a complete normal subgroup of an $n$-dimensional space group $\Gamma$, 
and let $V = {\rm Span}({\rm N})$.  
Then the flat orbifold $E^n/\Gamma$ geometrically fibers over the flat orbifold 
$(E^n/V)/(\Gamma/{\rm N})$ with generic fiber the flat orbifold $V/{\rm N}$. 
\end{theorem}
\begin{proof}
Let $\eta:E^n/\Gamma\to (E^n/V)/(\Gamma/{\rm N})$ be the map 
defined by 
$$\eta(\Gamma x) = (\Gamma/{\rm N})(V+x) = \Gamma(V+x).$$
Then the following diagram commutes
\[\begin{array}{ccc}
E^n & {\buildrel \tilde{\eta} \over \longrightarrow} & E^n/V \\
 \downarrow  & & \downarrow \\
E^n/\Gamma & {\buildrel \eta \over \longrightarrow} &(E^n/V)/(\Gamma/{\Nu}) 
  \end{array}
\] 
where $\tilde\eta$ and the vertical maps are the quotient maps. 
Hence $\eta$ is a continuous surjection, 
and so $(E^n/V)/(\Gamma/{\rm N})$ is a compact flat $m$-orbifold 
with $m = {\rm dim}(E^n/V)$.  

Let $x\in V^\perp$, and let 
$$\Gamma_{V+x} = \{\gamma\in\Gamma: \gamma(V+x) = V+x\}.$$
Then the stabilizer of $V+x$ in $\Gamma/{\rm N}$ is $G_x=\Gamma_{V+x}/{\rm N}$. 
By Lemma 1, the group $G_x$ is finite. 
Let $r$ be half the distance from $V+x$ 
to the set of remaining points of the orbit $(\Gamma/{\rm N})(V+x)$ in $E^n/V$. 
By Lemma 1, we have that $r>0$. 
Let $B_x$ be the open ball of radius $r$ centered at $V+x$ in $E^n/V$. 
The quotient map from $E^n/V$ to $(E^n/V)/(\Gamma/{\rm N})$ maps 
$B_x$ onto $(\Gamma/{\rm N})B_x/(\Gamma/{\rm N})$ which is isometric 
to $B_x/G_x$ with $G_x$ acting effectively and orthogonally on $B_x$. 

Let $U$ be the $r$-neighborhood of $V+x$ in $E^n$. 
Then we have 
$$\eta^{-1}((\Gamma/{\rm N})B_x/(\Gamma/{\rm N})) = \Gamma U/\Gamma.$$
Now $\Gamma U/\Gamma$ is isometric to $U/\Gamma_{V+x}$. 
Moreover we have
$$U/\Gamma_{V+x} = (U/{\rm N})/(\Gamma_{V+x}/{\rm N}) = (U/{\rm N})/G_x.$$ 
The group ${\rm N}$ acts trivially on $E^n/V$. 
Hence $U/{\rm N}$ is isometric to $\big((V+x)/{\rm N}\big)\times B_x$. 
Let $F= V/{\rm N}$.  Then $F$ is a compact flat $(n-m)$-orbifold. 
Now $F$ is isometric to $(V+x)/{\rm N}$, since ${\rm N}$ acts trivially on $E^n/V$.  
The finite group $G_x$ acts diagonally on 
$\big((V+x)/{\rm N}\big)\times B_x$, isometrically on $(V+x)/{\rm N}$, 
and effectively and orthogonally on $B_x$. 
We have a commutative diagram
\[\begin{array}{ccc}
\Gamma U/\Gamma & {\buildrel \ov\eta\over\longrightarrow} &(\Gamma/{\rm N})B_x/(\Gamma/{\rm N}) \\
 \downarrow  & & \downarrow \\
\big(\big((V+x)/{\rm N}\big)\times B_x\big)/G_x &  \longrightarrow & B_x/G_x  
\end{array}\] 
where the vertical maps are isometries, $\ov\eta$ is the restriction of $\eta$, 
and the bottom map is the natural projection. 
Thus $E^n/\Gamma$ geometrically fibers over the flat $m$-orbifold $(E^n/V)/(\Gamma/{\rm N})$ 
with generic fiber the flat $(n-m)$-orbifold $F=V/{\rm N}$. 
\end{proof}

Let ${\rm N}$ be a normal subgroup of an $n$-dimensional space group $\Gamma$.  
We call the map $\eta:E^n/\Gamma\to (E^n/V)/(\Gamma/\ov{\rm N})$ 
defined in the proof of Theorem 4, the {\it fibration projection determined by} ${\rm N}$. 
Let ${\rm T} = \{a+I\in\Gamma\}$. 
By Theorem 4, the fibration projection $\eta$ determined by $\Nu$ 
is an injective Seifert fibering, with typical fiber $V/({\rm N}\cap{\rm T})$,  
in the sense of Lee and Raymond \cite{L-R}.

\begin{theorem} 
Let ${\rm N}$ be a normal subgroup of a space group $\Gamma$. 
Then the following are equivalent: 
\begin{enumerate}
\item The quotient group $\Gamma/{\rm N}$ is a space group. 
\item The quotient group $\Gamma/{\rm N}$ has no nontrivial finite normal subgroups. 
\item The normal subgroup ${\rm N}$ of $\Gamma$ is complete.
\end{enumerate}
\end{theorem}
\begin{proof} 
By Theorem 2 every normal subgroup of a space 
group is a space group.  
Hence a space group has no nontrivial finite normal subgroups. 
Therefore (1) implies (2). 
As $\ov{\rm N}/{\rm N}$ is a finite normal subgroup of $\Gamma/{\rm N}$, 
we have that (2) implies (3). 
Finally (3) implies (1) by Theorem 4. 
\end{proof}

If $\Gamma$ is a group, 
let $Z(\Gamma)$ be the {\it center} of $\Gamma$. 
If $\Gamma$ is a finitely generated group, 
let $\beta_1$ be the {\it first Betti} number of $\Gamma$. 
Here $\Gamma/[\Gamma,\Gamma] \cong G\oplus \integers^{\beta_1}$ 
with $G$ a finite abelian group. 
The next theorem strengthens Theorem 6 of Farkas \cite{Farkas}.

\begin{theorem} 
If $\Gamma$ is a space group,    
then every element of $Z(\Gamma)$ is a translation, 
the rank of $Z(\Gamma)$ is $\beta_1$, 
and $Z(\Gamma)$ is a complete normal subgroup of $\Gamma$.  
\end{theorem}
\begin{proof}
The translations in $\Gamma$ are characterized 
as the elements of $\Gamma$ with only finitely many conjugates. 
Hence $Z(\Gamma)$ is a subgroup of the group $\Tau$ 
of translations of $\Gamma$. 

Let $\Pi$ be the point group of $\Gamma$. 
If $a+I\in \Tau$ and $b+B\in\Gamma$, then 
$$(b+B)(a+I)(b+B)^{-1} = Ba+I.$$
Hence conjugation in $\Gamma$ induces an action of $\Pi$ on $\Tau$. 
The group extension 
$$ 1 \to \Tau \longrightarrow \Gamma \longrightarrow \Pi \to 1$$
determines an exact sequence of cohomology groups 
$$H^1(\Pi,\integers) \to H^1(\Gamma, \integers)\to H^1(\Tau,\integers)^\Pi\to H^2(\Pi,\integers).$$
Here $\Tau, \Gamma$ and $\Pi$ act trivially on $\integers$ 
and $H^1(\Tau,\integers)^\Pi$ is the subgroup of $H^1(\Tau,\integers)$ 
of elements fixed under the induced action of $\Pi$. 
By the universal coefficients theorem, 
$H^1(\Pi,\integers)\cong {\rm Hom}(H_1(\Pi),\integers) = 0$ 
and $\beta_1 ={\it rank}(H^1(\Gamma,\integers))$. 
By Corollary 5.5 in Chapter IV of Mac Lane \cite{Mac}, we have  
$H^2(\Pi,\integers) \cong {\rm Hom}(\Pi,S^1)$, and so $H^2(\Pi,\integers)$ is finite. 
Hence $\beta_1 = {\it rank}(H^1(\Tau,\integers)^\Pi)$. 
By the universal coefficients theorem, 
$$H^1(\Tau,\integers)^\Pi \cong {\rm Hom}(H_1(\Tau),\integers)^\Pi \cong \Tau^\Pi = Z(\Gamma).$$
Thus ${\it rank}(Z(\Gamma)) = \beta_1$. 

Let $V={\rm Span}(Z(\Gamma))$. 
If $a+I\in Z(\Gamma)$ and $b+B\in\Gamma$, then 
$$Ba+I =(b+B)(a+I)(b+B)^{-1} = a+I.$$
Hence $V\subseteq {\rm Fix}(B)$. 
Therefore $\ov{Z(\Gamma)}=\{a+I \in \Gamma: a \in V\} = Z(\Gamma)$. 
\end{proof}

Let $Z(\Gamma)$ be the center of an $n$-dimensional space group $\Gamma$. 
Then every element of $Z(\Gamma)$ is a translation, 
the rank of $Z(\Gamma)$ is the first Betti number $\beta_1$ of $\Gamma$, 
and $Z(\Gamma)$ is a complete normal subgroup of $\Gamma$ by Theorem 6. 
Let $V={\rm Span}(Z(\Gamma))$. 
By Theorem 4, the flat orbifold $E^n/\Gamma$ geometrically fibers 
over the flat orbifold $(E^n/V)/(\Gamma/Z(\Gamma))$ 
with generic fiber the flat $\beta_1$-torus $V/Z(\Gamma)$. 

Suppose $x\in V^\perp$ and $b+B\in\Gamma_{V+x}$. 
Write $b=c+d$ with $c\in V$ and $d\in V^\perp$. 
Then $(b+B)(V+x) = V+d+Bx$, and so $d+Bx = x$. 
If $v\in V$, then 
$$(b+B)(v+x) = b+v+Bx = c+v+x.$$
Thus $\Gamma_{V+x}$ acts as a discrete group 
of translations on $V+x$. 
As $\Gamma_{V+x}$ contains $Z(\Gamma)$, 
we have that $(V+x)/\Gamma_{V+x}$ is a $\beta_1$-torus. 
Thus all the fibers of the fibration projection 
$\eta: E^n/\Gamma\to (E^n/V)/(\Gamma/Z(\Gamma))$ are $\beta_1$-tori,

The $\beta_1$-torus $V/Z(\Gamma)$, as an additive group, 
acts effectively on $E^n/\Gamma$ by
$$(Z(\Gamma) v)(\Gamma x) = \Gamma(v+x).$$
The projection from $(V+x)/Z(\Gamma)$ to $(V+x)/\Gamma_{V+x}$ 
is a covering projection for each $x\in V^\perp$.  
Therefore the action of $V/Z(\Gamma)$ on $E^n/\Gamma$ is an injective toral action 
in the sense of Conner and Raymond \cite{C-R}.

\section{Equivalence of Geometric Fibrations}  

Let $M$ be a flat $n$-orbifold. 
Suppose $M$ geometrically fibers over a flat $m$-orbifold $B_i$ 
with generic fiber a flat $(n-m)$-orbifold $F_i$
and fibration projection $\eta_i:M\to B_i$ for $i=1,2$.  
Then the fibration projections $\eta_1$ and $\eta_2$ are said to be 
{\it geometrically equivalent} if there is an isometry $\beta:B_1\to B_2$ 
such that $\beta\eta_1=\eta_2$. 

\begin{theorem} 
Let $\Gamma$ be an $n$-dimensional space group. 
Suppose that the flat orbifold $E^n/\Gamma$ geometrically fibers over 
a flat $m$-orbifold $B$ with generic fiber a connected flat $(n-m)$-orbifold $F$ 
and fibration projection $\eta:E^n/\Gamma\to B$. 
Then $\Gamma$ has a complete normal subgroup ${\rm N}$ such that 
$\eta$ is geometrically equivalent to the fibration projection determined by ${\rm N}$.  
\end{theorem}
\begin{proof}
The fibration projection $\eta$ is locally isometrically equivalent to 
a natural projection $(F\times D)/G\to D/G$, 
where $G$ is a finite group acting diagonally on $F\times D$, isometrically on $F$, 
and effectively and orthogonally on an open $m$-disk $D$. 
The fibers of the projection $(F\times D)/G\to D/G$ are connected, totally geodesic, and parallel, 
and so the fibers of $\eta$ are connected, totally geodesic, and parallel. 
Hence there is a vector subspace $V$ of $E^n$ such that 
each coset of $V$ in $E^n$ projects onto a fiber of $\eta$ 
and $\Gamma$ maps each coset of $V$ to a coset of $V$. 
Let $b+B\in\Gamma$. Then $(b+B)V = b+BV$, and so $BV=V$.

Let $V+x$ project to a generic fiber $F$ of $\eta$, 
that is, to a fiber of $\eta$ with $G = 1$. 
By conjugating $\Gamma$ by $-x+I$, we may assume that $x=0$. 
Now $F = \Gamma V/\Gamma$ which is isomorphic to $V/\Gamma_V$ 
where $\Gamma_V = \{\gamma\in\Gamma: \gamma(V) = V\}$. 
As $\eta$ is isometrically equivalent to the projection $F\times D \to D$ 
in a tubular neighborhood of the fiber $F$, we deduce that  
$$\Gamma_V = \{a+A\in \Gamma: a\in V\ \hbox{and}\ V^\perp\subseteq{\rm Fix}(A)\}.$$

We claim that $\Gamma_V$ is a normal subgroup of $\Gamma$. 
Let $b+B\in\Gamma$ and $a+A\in \Gamma_V$. 
Then we have 
$$(b+B)(a+A)(b+B)^{-1} = b+Ba-BAB^{-1}b+BAB^{-1}.$$
If $x\in V^\perp$, then $B^{-1}x\in V^\perp$, and so $BAB^{-1}x = x$. 
Therefore $V^\perp \subseteq {\rm Fix}(BAB^{-1})$. 
Write $b=c+d$ with $c\in V$ and $d\in V^\perp$. 
Then we have 
\begin{eqnarray*}
b+Ba-BAB^{-1}b & = & b+Ba-BAB^{-1}c -BAB^{-1}d \\
& = & b+Ba-BAB^{-1}c -d \\
& = & c+Ba-BAB^{-1}c 
\end{eqnarray*}
which is an element of $V$. 
Therefore $(b+B)(a+A)(b+B)^{-1}\in \Gamma_V$. 
Thus $\Gamma_V$ is a normal subgroup of $\Gamma$. 

Now $F= V/\Gamma_V$ is compact, and so $\Gamma_V$ acts 
as a space group of isometries of $V$. 
Therefore $V = {\rm Span}(\Gamma_V)$ by Part (1) 
of Theorem 1. 
Hence $\ov\Gamma_V=\Gamma_V$. 

The fibration projection $\eta$ and the fibration projection $\eta_V$ determined by $\Gamma_V$ 
have the same fibers. 
Hence there is a homeomorphism $\beta:B\to (E^n/V)/(\Gamma/\Gamma_V)$ 
such that $\beta\eta = \eta_V$. 
The map $\phi$ is an isometry, since the metrics on $B$ and $(E^n/V)/(\Gamma/\Gamma_V)$ 
are determined by the distance between fibers in $E^n/\Gamma$. 
Therefore $\eta$ is geometrically equivalent to $\eta_V$. 
\end{proof}

Let $M_i$ be a connected, complete, flat $n$-orbifold for $i=1,2$. 
Suppose $M_i$ geometrically fibers over a flat $m$-orbifold $B_i$ 
with generic fiber a flat $(n-m)$-orbifold $F_i$ and fibration projection 
$\eta_i:M_i\to B_i$ for $i=1,2$. 
The fibration projections $\eta_1$ and $\eta_2$ are said to be 
{\it isometrically equivalent} if there are isometries $\alpha:M_1\to M_2$ 
and $\beta:B_1\to B_2$ such that $\beta\eta_1=\eta_2\alpha$.

\begin{theorem} 
Let $M$ be a compact, connected, flat $n$-orbifold. 
If $M$ geometrically fibers over a flat $m$-orbifold $B$ 
with generic fiber a flat $(n-m)$-orbifold $F$ and fibration projection $\eta:M\to B$,  
then there exists an $n$-dimensional space group $\Gamma$ 
and a complete normal subgroup $\Nu$ of $\Gamma$   
such that $\eta$ is isometrically equivalent to the fibration 
projection $\eta_V: E^n/\Gamma\to (E^n/V)/(\Gamma/\Nu)$ determined by $\Nu$. 
\end{theorem} 
\begin{proof}
There exists an $n$-dimensional space group $\Gamma$ 
and an isometry $\alpha: M \to E^n/\Gamma$ by Theorem 13.3.10 of Ratcliffe\cite{R}. 
The map $\eta\alpha^{-1}: E^n/\Gamma \to B$ is a geometric fibration projection. 
There exists a complete normal subgroup $\Nu$ of $\Gamma$ and an isometry 
$\beta: B\to (E^n/V)/(\Gamma/\Nu)$ such that $\beta(\eta\alpha^{-1}) = \eta_V$ by Theorem 7. 
Hence $\beta\eta=\eta_V\alpha$, and so $\eta$ is isometrically equivalent to $\eta_V$. 
\end{proof}

\begin{theorem} 
Let $\Nu_i$ be a complete normal subgroup of an $n$-dimensional 
crystal\-lographic group $\Gamma_i$ for $i=1,2$, 
and let $\eta_i:E^n/\Gamma_i \to (E^n/V_i)/(\Gamma_i/\Nu_i)$ 
be the fibration projections determined by $\Nu_i$ for $i=1,2$. 
Then $\eta_1$ and $\eta_2$ are isometrically equivalent if and only if 
there is an isometry $\xi$ of $E^n$ such that $\xi\Gamma_1\xi^{-1} = \Gamma_2$ 
and $\xi\Nu_1\xi^{-1} = \Nu_2$. 
\end{theorem}
\begin{proof}
Suppose  $\eta_1$ and $\eta_2$ are isometrically equivalent. 
Then there exists isometries $\alpha:E^n/\Gamma_1\to E^n/\Gamma_2$ 
and $\beta: (E^n/V_1)/(\Gamma/\Nu_1)\to(E^n/V_2)/(\Gamma/\Nu_2)$ 
such that $\beta\eta_1=\eta_2\alpha$. 
The isometry $\alpha$ lifts to an isometry $\xi$ of $E^n$ 
such that $\alpha(\Gamma_1 x) = \Gamma_2\xi(x)$ for each $x\in E^n$ 
by Theorem 13.2.6 of Ratcliffe \cite{R}. 
Therefore $\xi\Gamma_1\xi^{-1}=\Gamma_2$. 
As $\beta\eta_1=\eta_2\alpha$, the isometry $\alpha$ maps 
a generic fiber of $\eta_1$ onto a generic fiber of $\eta_2$. 
Let $V_1+x$ be a coset of $V_1$ that projects to a generic fiber of $\eta_1$. 
Then $\xi$ maps $V_1+x$ onto $V_2+\xi(x)$ 
and $V_2+\xi(x)$ projects to a generic fiber of $\eta_2$. 
Hence $\xi$ conjugates the stabilizer of $V_1+x$ in $\Gamma_1$ 
to the stabilizer of $V_2+\xi(x)$ in $\Gamma_2$. 
Therefore $\xi\Nu_1\xi^{-1} = \Nu_2$. 

Conversely, suppose $\xi$ is an isometry of $E^n$ 
such that $\xi\Gamma_1\xi^{-1} = \Gamma_2$ 
and $\xi\Nu_1\xi^{-1} = \Nu_2$. 
Then $\xi$ induces an isometry $\alpha: E^n/\Gamma_1 \to E^n/\Gamma_2$ 
defined by $\alpha(\Gamma_1x) = \Gamma_2\xi(x)$ for each $x\in E^n$. 
Let $\Tau_i$ be the group of translations of $\Gamma_i$ for $i=1,2$. 
Then $\xi\Tau_1\xi^{-1} = \Tau_2$, since $\Tau_i$ is the 
unique maximal abelian normal subgroup of $\Gamma_i$ for $i=1,2$.  
Hence $\xi(\Nu_1\cap\Tau_1)\xi^{-1} = \Nu_2\cap\Tau_2$,  
and so if $\xi =b+B$, then $BV_1 = V_2$. 
Therefore $\xi$ maps each coset of $V_1$ in $E^n$ onto a coset of $V_2$ in $E^n$.  
Hence $\xi$ induces an isometry $\ov\xi: E^n/V_1\to E^n/V_2$ 
defined $\ov\xi(V_1+x) = V_2+\xi(x)$. 
If $\gamma\in \Gamma_1$ and $x\in E^n$, then we have
\begin{eqnarray*}
\ov\xi(\Nu_1\gamma(V_1+x)) & = & \ov\xi(V_1+\gamma(x)) \\
 & = & V_2+\xi\gamma(x) \\
 & = & V_2 + \xi\gamma\xi^{-1}\xi(x) \ \
  = \ \  (\Nu_2\xi\gamma\xi^{-1})\ov\xi(V_1+x).
 \end{eqnarray*}
Hence $\ov\xi(\Nu_1\gamma) = (\Nu_2\xi\gamma\xi^{-1})\ov\xi$, 
and so $\ov\xi(\Nu_1\gamma)\ov\xi^{-1}=\Nu_2\xi\gamma\xi^{-1}$. 
Therefore we have $\ov\xi(\Gamma_1/\Nu_1)\ov\xi\hbox{}^{-1} = \Gamma_2/\Nu_2$. 
Hence $\ov\xi$ induces an isometry
$$\beta: (E^n/V_1)/(\Gamma_1/\Nu_1)\to(E^n/V_2)/(\Gamma_2/\Nu_2)$$
defined by 
$$\beta((\Gamma_1/\Nu_1)(V_1+x)) = (\Gamma_2/\Nu_2)\ov\xi(V_1+x) = (\Gamma_2/\Nu_2)(V_2+\xi(x)).$$
If $x\in E^n$, we have 
\begin{eqnarray*}
\beta\eta_1(\Gamma_1x) & = & \beta((\Gamma_1/\Nu_1)(V_1+x)) \\
 & = & (\Gamma_2/\Nu_2)(V_2+\xi(x)) \\
 & = & \eta_2(\Gamma_2\xi(x)) \ \ = \ \ \eta_2\alpha(\Gamma_1x).
\end{eqnarray*}
Therefore $\beta\eta_1=\eta_2\alpha$. 
Thus $\eta_1$ and $\eta_2$ are isometrically equivalent. 
\end{proof}

Let $M$ be a connected, complete flat $n$-orbifold. 
A {\it universal orbifold covering projection} is 
a geometric fibration projection $\pi: \tilde M \to M$  
that is isometrically equivalent to the natural projection $\pi_\Gamma:E^n\to E^n/\Gamma$ 
for some discrete group $\Gamma$ of isometries of $E^n$. 
There exists a universal orbifold covering projection $\pi: \tilde M \to M$ 
by Theorem 13.3.10 of \cite{R}, and
any two universal orbifold covering projections $\pi_i:\tilde M_i \to M$, $i=1,2$, 
are isometrically equivalent by Theorem 13.2.6 of \cite{R}.

Let $\pi_i:\tilde M_i\to M_i$ be universal orbifold covering projections for $i=1,2$. 
A homeomorphism $\alpha: M_1\to M_2$ is said to be {\it affine} 
if there is an affine homeomorphism $\tilde\alpha: \tilde M_1 \to \tilde M_2$ 
such that $\alpha\pi_1=\pi_2\tilde\alpha$. 
Note that $\alpha: M_1\to M_2$ being affine does not depend 
on the choice of the universal orbifold covering projections $\pi_i:\tilde M_i\to M_i$.

The fibration projections $\eta_1$ and $\eta_2$ are said to be 
{\it affinely equivalent} if there are affine homeomorphisms $\alpha:M_1\to M_2$ 
and $\beta:B_1\to B_2$ such that $\beta\eta_1=\eta_2\alpha$. 

\begin{theorem} 
Let $\Nu_i$ be a complete normal subgroup of an $n$-dimensional 
space group $\Gamma_i$ for $i=1,2$, 
and let $\eta_i:E^n/\Gamma_i \to (E^n/V_i)/(\Gamma_i/\Nu_i)$ 
be the fibration projections determined by $\Nu_i$ for $i=1,2$. 
Then the following are equivalent:
\begin{enumerate}
\item The fibration projections $\eta_1$ and $\eta_2$ are affinely equivalent. 
\item There is an affine homeomorphism $\phi$ of $E^n$ such that 
$\phi\Gamma_1\phi^{-1} = \Gamma_2$ and $\phi\Nu_1\phi^{-1} = \Nu_2$. 
\item There is an isomorphism $\psi:\Gamma_1\to \Gamma_2$ such that $\psi(\Nu_1) = \Nu_2$.
\end{enumerate}
\end{theorem}
\begin{proof}
The proof of the equivalence of (1) and (2) is the same as the proof of Theorem 9. 
The equivalence of (2) and (3) follows from Theorem 7.5.4 of Ratcliffe \cite{R}. 
\end{proof}

\section{Reducibility of Crystallographic Groups} 

Let ${\rm N}$ be a normal subgroup of an $n$-dimensional 
space group $\Gamma$. 
The {\it dimension} of $\Nu$, denoted by ${\rm dim}(\Nu)$,
is defined to be the dimension of $V = {\rm Span}(\Nu)$. 
Note that ${\rm dim}(\Nu)$ is equal to the virtual cohomological dimension 
of $\Nu$ by Theorem 2. 
The theory in this paper is nontrivial only if $0 < {\rm dim}(\Nu) < n$, 
and so we should discuss when such a normal subgroup ${\rm N}$ exists. 

Let ${\rm T}$ be the group of translations of $\Gamma$,  
and let $\Pi$ be the point group of $\Gamma$. 
The group ${\rm T}$ is free abelian of rank $n$ 
and $\{b: b+I\in {\rm T}\}$ spans $E^n$ by Theorem 1. 
Choose a set $\{b_1+I,\ldots,b_n+I\}$ of $n$ generators of ${\rm T}$.  
Then $\{b_1,\ldots, b_n\}$ is a basis of $E^n$. 
Let $\gamma = b+B \in \Gamma$. 
Then $\gamma(b_j+I)\gamma^{-1} = Bb_j+I\in{\rm T}$ for each $j=1,\ldots,n$. 
Hence there are integers $c_{ij}$ for $i,j=1,\ldots,n$
such that $Bb_j = \sum_{i=1}^nc_{ij}b_i$ for each $j$. 
The representation $\rho:\Pi\to {\rm GL}(n,\mathbb Z)$ 
defined by $\rho(B) = (c_{ij})$ is a monomorphism. 
The representation $\rho$ is said be {\it reducible} 
if there is an integer $k$, with $0< k < n$, such that 
every matrix in the image of $\rho$ is in the block form 
$$\left(\begin{array}{cc}
A & B \\
O & D 
\end{array}\right)$$
where $A$ is a $k\times k$ block and $O$ is a $(n-k)\times k$ block of zeros. 

The group $\Gamma$ is said to be $\mathbb Z$-{\it reducible}, 
if there is a set of $n$ generators of ${\rm T}$ 
such that the corresponding representation $\rho:\Pi\to {\rm GL}(n,\mathbb Z)$ is reducible. 

\begin{theorem} 
Let $\Gamma$ be an $n$-dimensional space group 
with translation group ${\rm T}$ and point group $\Pi$.  
Then the following are equivalent: 
\begin{enumerate}
\item The group $\Gamma$ is $\mathbb Z$-reducible. 
\item There exists a $\Pi$-invariant vector subspace $V$ of $E^n$ 
with basis $\{b_1,\ldots, b_k\}$ such that  
$0< k < n$ and $b_i+I \in {\rm T}$ for each $i=1,\ldots, k$.  
\item The group $\Gamma$ has a normal subgroup ${\rm N}$ 
of dimension $k$ with $0 < k < n$. 
\end{enumerate} 
\end{theorem}
\begin{proof}
Suppose $\Gamma$ is $\mathbb Z$-reducible. 
Then there is a set of generators $\{b_1+I,\ldots,b_n+I\}$ 
of ${\rm T}$ such that the corresponding representation 
$\rho:\Pi\to {\rm GL}(n,\mathbb Z)$ is reducible. 
Let $k$ be the integer in the definition of reducibility, 
and let $V={\rm Span}\{b_1,\ldots, b_k\}$. 
Then $V$ is a $\Pi$-invariant vector subspace of $E^n$ 
with basis $\{b_1,\ldots, b_k\}$ such that  
$0< k < n$ and $b_i+I \in {\rm T}$ for each $i=1,\ldots, k$. 
Thus (1) implies (2). 

Suppose there exists a $\Pi$-invariant vector subspace $V$ of $E^n$\! 
with basis $\{b_1,\ldots, b_k\}$ such that  
$0< k < n$ and $b_i+I \in {\rm T}$ for each $i=1,\ldots, k$.  
Define 
$${\rm N}=\{a+I\in{\rm T}:\ a\in V\}.$$ 
As $V$ is $\Pi$-invariant, ${\rm N}$ is a normal subgroup of $\Gamma$. 
As $V = {\rm Span}(\Nu)$, we have that ${\rm dim}(\Nu) = k$.   
Thus (2) implies (3). 

Suppose $\Gamma$ has a normal subgroup ${\rm N}$ of dimension $k$ 
with $0 < k < n$.  
Then ${\rm T}$ has a set of generators $\{b_1+I,\ldots, b_n+I\}$ 
such that $b_1+I,\ldots, b_k+I$ generate $\ov{\Nu}\cap{\rm T}$ 
by the proof of Lemma 1. 
As $\ov{\rm N}\cap{\rm T}$ is a normal subgroup of $\Gamma$, 
the additive group generated by $b_1,\ldots, b_k$ is $\Pi$-invariant. 
Hence the representation $\rho:\Pi\to {\rm GL}(n,\mathbb Z)$ 
determined by the set of generators $\{b_1+I,\ldots,b_n+I\}$ of ${\rm T}$ 
is reducible. 
Thus (3) implies (1). 
\end{proof}

\begin{theorem} 
Let $\Gamma$ be an $n$-dimensional space group. 
If $\Gamma$ has a normal subgroup ${\rm N}$ of dimension $k$, 
then $\Gamma$ has a normal subgroup ${\rm N}'$ of dimension $(n-k)$. 
\end{theorem}
\begin{proof}
This follows from Theorem 11 and Proposition 2.1.1 of 
Brown, Neub\"user, and Zassenhaus \cite{B-N-Z}. 
\end{proof}

\section{Geometric Fiber Bundles} 

Let $\Gamma$ be an $n$-dimensional space group. 
If $E^n/\Gamma$ geometrically fibers over a flat manifold $B$ 
with generic fiber a flat orbifold $F$, 
then $E^n/\Gamma$ is a fiber bundle over $B$ with totally geodesic 
fibers isometric to $F$.  

\vspace{.15in}
\noindent{\bf Definition:}
A flat $n$-orbifold $M$ is a {\it geometric fiber bundle} over a flat $m$-manifold $B$ 
with fiber a flat $(n-m)$-orbifold $F$ if there is a surjective map $\eta: M \to B$, 
called the {\it fibration projection},  
such that for each point $y$ in $B$, there is an open metric ball $B(y,r)$ 
of radius $r>0$ centered at $y$ in $B$ such that 
$\eta$ is isometrically equivalent on $\eta^{-1}(B(y,r))$ 
to the natural projection $F\times B_y \to B_y$ onto an open metric ball $B_y$ in $E^m$ 
of radius $r$. 

\vspace{.15in}

\begin{lemma} 
Let ${\rm N}$ be a complete normal subgroup of an $n$-dimensional space group $\Gamma$, 
and let $V = {\rm Span}({\rm N})$. 
Then $(E^n/V)/(\Gamma/{\rm N})$ is a flat manifold 
if and only if $\Gamma/{\rm N}$ is torsion-free. 
\end{lemma}
\begin{proof}
Suppose  $(E^n/V)/(\Gamma/{\rm N})$ is a flat manifold. 
Then the quotient map from $E^n/V$ to $(E^n/V)/(\Gamma/{\rm N})$ 
is a universal covering projection with $\Gamma/{\rm N}$ its group 
of covering transformations.  
Therefore $(E^n/V)/(\Gamma/{\rm N})$ is an aspherical manifold, 
and so its fundamental group is torsion-free. 
Therefore $\Gamma/{\rm N}$ is torsion-free. 

Conversely if $\Gamma/{\rm N}$ is torsion-free, 
then $(E^n/V)/(\Gamma/{\rm N})$ is a flat manifold, 
since the finite group $G_x=\Gamma_{V+x}/{\rm N}$ is trivial for each $x\in V^\perp$. 
\end{proof}

\begin{theorem} 
Let $\Nu$ be a normal subgroup of an $n$-dimensional space group $\Gamma$ 
such that $\Gamma/\Nu$ is torsion-free, and let $V = {\rm Span}(\Nu)$. 
Then $\Nu$ is complete, and the flat orbifold $E^n/\Gamma$ is a geometric fiber bundle 
over the flat manifold $(E^n/V)/(\Gamma/\Nu)$ with fiber the flat orbifold $V/\Nu$. 
\end{theorem}
\begin{proof}
We have that $\Nu$ is complete by Theorem 5. 
The rest of the theorem follows from Lemma 2 and Theorem 4. 
\end{proof}

\begin{theorem} 
Let $\Gamma$ be an $n$-dimensional space group 
with first Betti number $\beta_1$. 
Then $\Gamma$ has a unique normal subgroup ${\rm N}$ such that $\Gamma/{\rm N}$ 
is a free abelian group of rank $\beta_1$,  
and the flat orbifold $E^n/\Gamma$ uniquely fibers as a geometric fiber bundle 
over a flat $\beta_1$-torus. 
\end{theorem}
\begin{proof}
We have $\Gamma/[\Gamma, \Gamma] \cong G\oplus\mathbb Z^{\beta_1}$ with $G$ a finite abelian group. 
Hence the subgroup of $\Gamma$ containing $[\Gamma,\Gamma]$ that corresponds to $G$ 
is the unique normal subgroup ${\rm N}$ of $\Gamma$ such that $\Gamma/{\rm N}$ 
is a free abelian group of rank $\beta_1$. 
By Theorem 13, we have that ${\rm N}$ is complete. 
Let $V={\rm Span}({\rm N})$. 
Then $(E^n/V)/(\Gamma/{\rm N})$ is a flat $\beta_1$-torus, 
since $\Gamma/{\rm N}$ is a free abelian group of rank $\beta_1$. 
Therefore $E^n/\Gamma$ fibers as a geometric fiber bundle over a flat $\beta_1$-torus by Theorem 4. 
The flat orbifold $E^n/\Gamma$ uniquely fibers as a geometric fiber bundle 
over a flat $\beta_1$-torus by Theorem 7, since ${\rm N}$ is the unique normal 
subgroup of $\Gamma$ such that $\Gamma/{\rm N}$ is free abelian of rank $\beta_1$.  
\end{proof}

\begin{lemma} 
If $\Gamma$ is an $n$-dimensional space group with translation group {\rm T} and 
point group $\Pi$,  
then the transfer homomorphism ${\rm tr}: \Gamma \to {\rm T}$ 
is given by the formula
$${\rm tr}(b+B) = \left({\textstyle\sum}\{A:A\in\Pi\}\right)b+ I \ \ \hbox{for each}\ \ b+B\in\Gamma.$$
\end{lemma}
\begin{proof}
For each $A\in\Pi$, choose a coset representative $a_A+A$
of T in $\Gamma$ corresponding to $A$. 
Given an element $b+B\in \Gamma$ and a coset representative $a_A+A$,  
then there exists a unique coset representative $a_{A'}+A'$ 
such that 
$$(a_A+A)(b+B)(a_{A'}+A')^{-1}\in{\rm T}.$$
The transfer homomorphism ${\rm tr}: \Gamma \to {\rm T}$ is defined by the formula
$${\rm tr}(b+B) = {\textstyle\prod}\{(a_A+A)(b+B)(a_{A'}+A')^{-1}: A\in \Pi\}.$$
We have that 
$$(a_A+A)(b+B)(a_{A'}+A')^{-1} = a_A+Ab- AB(A')^{-1}a_{A'}+AB(A')^{-1}.$$
Therefore $AB(A')^{-1}=I$, and so $A' = AB$. 
Hence we have that
\begin{eqnarray*}
{\rm tr}(b+B) & = & {\textstyle\prod}\{a_A+Ab-a_{AB}+I: A\in\Pi\} \\
			   & = & \big({\textstyle\sum}\{a_A:A\in\Pi\}
			   -{\textstyle\sum}\{a_{AB}:A\in\Pi\}+{\textstyle\sum}\{Ab:A\in\Pi\}\big) + I \\
			   & = & \big({\textstyle\sum}\{A:A\in\Pi\}\big)b + I. 
\end{eqnarray*}
\end{proof}

\begin{lemma} 
If $\Pi$ is a finite group of orthogonal transformations of $E^n$,  
then 
\begin{enumerate}
\item ${\rm Im}({\textstyle\sum}\{A\in \Pi\}) = {\rm Fix}(\Pi)$, 
\item ${\rm Ker}({\textstyle\sum}\{A\in \Pi\}) = {\rm Fix}(\Pi)^\perp$. 
\end{enumerate}
\end{lemma}
\begin{proof}
Let $B\in \Pi$. 
Observe that 
$$(B-I)({\textstyle\sum}\{A\in \Pi\}) = O.$$
Now ${\rm Ker}(B-I) = {\rm Fix}(B)$. 
Hence we have
$${\rm Im}({\textstyle\sum}\{A\in \Pi\}) \subseteq {\rm Fix}(\Pi).$$
Let $x\in {\rm Fix}(\Pi)$. 
Then we have 
$$({\textstyle\sum}\{A\in \Pi\})(x) = |\Pi|x.$$
Hence $x\in {\rm Im}({\textstyle\sum}\{A\in \Pi\})$. 
This proves (1). 

Let $x\in {\rm Ker}({\textstyle\sum}\{A\in \Pi\})$. 
Write $ x = u+v$ with $u \in {\rm Fix}(\Pi)$ and $v\in {\rm Fix}(\Pi)^\perp$. 
Then we have that 
$$0 = ({\textstyle\sum}\{A\in \Pi\})(x) = |\Pi|u + ({\textstyle\sum}\{A\in \Pi\})(v).$$
Now $|\Pi|u \in {\rm Fix}(\Pi)$ and 
$({\textstyle\sum}\{A\in \Pi\})(v) \in {\rm Fix}(\Pi)^\perp$, 
since ${\rm Fix}(\Pi)^\perp$ is a $\Pi$-invariant subspace of $E^n$. 
Therefore $|\Pi|u = 0$, and so $u = 0$. 
Hence $x\in{\rm Fix}(\Pi)^\perp$. 

Conversely, suppose $x\in {\rm Fix}(\Pi)^\perp$. 
By (1), we have 
$$({\textstyle\sum}\{A\in \Pi\})(x)\in{\rm Fix}(\Pi)\cap {\rm Fix}(\Pi)^\perp = \{0\}.$$
Therefore $x\in {\rm Ker}({\textstyle\sum}\{A\in \Pi\})$. 
This proves (2). 
\end{proof}

Let $\Nu$ and $\Nu'$ be a normal subgroups of a space group $\Gamma$. 
We say that $\Nu$ and $\Nu'$ are {\it orthogonal} if ${\rm Span}(\Nu') = ({\rm Span}(\Nu))^\perp$. 
If $\Nu$ and $\Nu'$ are orthogonal, complete, normal subgroups of $\Gamma$, 
we define $\Nu^\perp = \Nu'$.  

\begin{theorem} 
Let $\Gamma$ be an $n$-dimensional space group 
with first Betti number $\beta_1$. 
Then the kernel of the transfer homomorphism $tr: \Gamma \to {\rm T}$ 
is the unique normal subgroup ${\rm N}$ of $\Gamma$ such that $\Gamma/{\rm N}$ 
is a free abelian group of rank $\beta_1$,  
Moreover ${\rm N}$ and $Z(\Gamma)$ are orthogonal, complete, normal subgroups of $\Gamma$.  
\end{theorem}
\begin{proof}
Let $\Pi$ be the point group of $\Gamma$. 
By Lemmas 3 and 4, we have that 
$$tr(Z(\Gamma)) = \{|\Pi|b+ I: b+I\in Z(\Gamma)\}\subseteq {\rm Im}(tr) \subseteq Z(\Gamma).$$
Hence ${\rm Im}(tr)$ is a free abelian group of rank $\beta_1$ by Theorem 6.   
Therefore ${\rm Ker}(tr)$ is the unique normal subgroup ${\rm N}$ of $\Gamma$ 
such that $\Gamma/{\rm N}$ is a free abelian group of rank $\beta_1$. 

By Lemmas 3 and 4, we have that 
$$\Nu = {\rm Ker}(tr) = \{b+ B\in \Gamma: b\in {\rm Fix}(\Pi)^\perp\}.$$
Hence 
${\rm Span}(\Nu) \subseteq {\rm Fix}(\Pi)^\perp.$
By Theorems 6 and 14, we have that 
$${\rm dim}({\rm Span}(\Nu)) = n-\beta_1 = {\rm dim}({\rm Fix}(\Pi)^\perp).$$
Therefore ${\rm Span}(\Nu) ={\rm Fix}(\Pi)^\perp$. 
Now as ${\rm Span}(Z(\Gamma)) ={\rm Fix}(\Pi)$, we conclude that $\Nu^\perp = Z(\Gamma)$. 
\end{proof}

\section{Crystallographic Group Extensions} 
Let $\Nu$ be a complete normal subgroup of an $n$-dimensional space group $\Gamma$, 
let $V = {\rm Span}(\Nu)$, and let $m = {\rm dim}(E^n/V)$.  
Then $\Nu$ is an $(n-m)$-dimensional space group by Theorem 2, and 
$\Gamma/\Nu$ is an $m$-dimensional space group by Theorem 4. 
We call the exact sequence of natural group homomorphisms, 
$$ 1 \to {\Nu}\ {\buildrel i\over \longrightarrow}\ \Gamma\ {\buildrel p\over\longrightarrow} 
\ \Gamma/\Nu\to 1,$$
a {\it space group extension}. 
In this section, we study the relationship between the point groups of $\Nu, \Gamma$, 
and $\Gamma/\Nu$.

\begin{lemma} 
Let $\Gamma$ be a space group with group of translations $\Tau$ and point group $\Pi$.  
Let $\Nu$ be a normal subgroup of $\Gamma$, 
and let $\Phi = \{A\in \Pi:a+A\in\Nu\ \hbox{for some}\ a\}$.
Then 
\begin{enumerate}
\item the group of translations of $\Nu$ is $\Nu\cap\Tau$, 
\item the point group of $\Nu$ is isomorphic to $\Phi$, and   
\item the group $\Phi$ is a normal subgroup of $\Pi$. 
\end{enumerate}
\end{lemma}
\begin{proof}
Let $V = {\rm Span}(\Nu)$, and let $a+A\in \Nu$. 
By Theorem 2, we have that $a\in V$ and $V^\perp \subseteq {\rm Fix}(A)$. 
Hence $a+A\in\Nu$ acts as a translation on $V$ if and only if $A = I$.  
Thus (1) holds. Let $\eta:\Gamma\to\Pi$ be defined by $\eta(b+B) = B$. 
Then (2) follows from (1), since the ${\rm Ker}(\eta|_\Nu)=\Nu\cap\Tau$, 
and (3) follows from (1), since $\Nu/(\Nu\cap\Tau) \cong\Nu\Tau/\Tau$, 
and so $\Phi$ corresponds to the normal subgroup $\Nu\Tau/\Tau$ of $\Gamma/\Tau$.  
\end{proof}

\begin{theorem} 
Let $\Gamma$ be a space group with point group $\Pi$.   
Let $\Nu$ be a complete normal subgroup of $\Gamma$, 
let $V = {\rm Span}(\Nu)$, let  $\Psi = \{B\in\Pi: V^\perp\subseteq {\rm Fix}(B)\}$, 
let $\Mu = \{b+B\in\Gamma: B\in \Psi\}$, 
and let $\Phi$ be as in Lemma 5.  
Then 
\begin{enumerate}
\item the group of translations of $\Gamma/\Nu$ is $\Mu/\Nu$, 
\item the group $\Phi$ is a normal subgroup of $\Psi$,   
\item the group $\Psi$ is a normal subgroup of $\Pi$, 
\item the point group of $\Gamma/\Nu$ is isomorphic to $\Pi/\Psi$.  
\end{enumerate}
\end{theorem}
\begin{proof} Let $b+B\in\Gamma$.  Suppose $b+B$ acts as a translation on $E^n/V$.  
Then $(b+B)(V+x) = V + c + x$ for some $c\in E^n$ for all $x\in V^\perp$. 
Now $(b+B)(V+x) = V+b+Bx$.  Taking $x=0$, we see that $b-c\in V$, 
and so $V+b+Bx = V+b+x$.  Hence $Bx-x \in V\cap V^\perp = \{0\}$. 
Therefore $x\in {\rm Ker}(B-I) = {\rm Fix}(B)$. 
Hence $V^\perp \subseteq {\rm Fix}(B)$, and so $b+B\in M$. 
Conversely if $b+B\in M$, then $b+B$ obviously acts as a translation on $E^n/V$. 
Thus (1) holds. 
Let $A\in \Phi$.  By Theorem 2, we have that $A\in\Psi$. 
Thus (2) holds by Lemma 5. 

Let $B\in \Psi$, and let $C\in\Pi$. By Theorem 2, we have that 
$C$ leaves $V$ invariant. Hence $C$ leaves $V^\perp$ invariant. 
Therefore $CBC^{-1}\in\Psi$. 
Thus (3) holds. 

Let $\Tau$ be the translation group of $\Gamma$. Then (4) follows from (1), since
$$(\Gamma/\Nu)/(\Mu/\Nu)\cong \Gamma/\Mu \cong (\Gamma/\Tau)/(\Mu/\Tau) \cong \Pi/\Psi.$$

\vspace{-.2in}
\end{proof}
 
For example, let $e_1 = (1,0)$ and $e_2 = (0,1)$, let $t_i  = e_i+I$ for $i=1,2$, 
let $\beta = \frac{1}{2}e_1+{\rm diag}(1,-1)$, and let $\Gamma = \langle t_1,t_2,\beta\rangle$.  
Then $\Gamma$ is a 2-dimensional space group, and $E^2/\Gamma$ is a flat Klein bottle. 
The group $\langle t_2\rangle$ is a complete normal subgroup of $\Gamma$ 
and ${\rm Span}\langle t_2\rangle = {\rm Span}\{e_2\}$. 
Hence $\Phi$ is trivial and $\Psi$ has order two. 

\begin{corollary} 
If $\Gamma$ is a space group with translation group $\Tau$,  
then the group of translations of $\Gamma/Z(\Gamma)$ is $\Tau/Z(\Gamma)$ 
and the point group of $\Gamma/Z(\Gamma)$ is isomorphic to the point group of $\Gamma$. 
\end{corollary}

\section{Splitting Crystallographic Group Extensions} 

Let $\Nu$ be a complete normal subgroup of an $n$-dimensional space group $\Gamma$, 
let $V = {\rm Span}(\Nu)$, and consider the corresponding space group extension
$$ 1 \to {\rm N}\ {\buildrel i\over \longrightarrow}\ \Gamma\ {\buildrel p\over\longrightarrow} 
\ \Gamma/\Nu\to 1.$$
In this section, we study the relationship between the above space group extension 
splitting ($p$ having a right inverse) and the corresponding fibration projection 
$\eta: E^n/\Gamma \to (E^n/V)/(\Gamma/\Nu)$ having an affine section. 
Here an {\it affine section} of $\eta$ is an affine map 
$\sigma: (E^n/V)/(\Gamma/\Nu) \to E^n/\Gamma$ such that $\eta\sigma$ is the identity map 
of $(E^n/V)/(\Gamma/\Nu)$. A map $\sigma: (E^n/V)/(\Gamma/\Nu) \to E^n/\Gamma$ is {\it affine} 
if $\sigma$ lifts to an affine map $\tilde\sigma: E^n/V \to E^n$ with respect to the 
natural quotient maps. 

\begin{lemma} 
Let $\Nu$ be a complete normal subgroup of an $n$-dimensional space group $\Gamma$, 
and let $V = {\rm Span}(\Nu)$. 
Let $\eta: E^n/\Gamma \to (E^n/V)/(\Gamma/\Nu)$ be the fibration projection determined by $\Nu$. 
If the space group extension 
$$ 1 \to {\rm N}\ {\buildrel i\over \longrightarrow}\ \Gamma\ {\buildrel p\over\longrightarrow} 
\ \Gamma/\Nu\to 1$$
splits, then $\eta$ has an affine section. 
\end{lemma}
\begin{proof}
Suppose that the space group extension splits. 
Then $\Gamma$ has a subgroup $\Sigma$ such that $p$ maps $\Sigma$ isometrically onto $\Gamma/\Nu$. 
By Theorem 5.4.6 of \cite{R}, the group $\Sigma$ has a free abelian subgroup $\Eta$ 
of rank $m$ and finite index, there is an $m$-plane $Q$ of $E^n$ such that 
$\Eta$ acts effectively on $Q$ as a discrete group of translations, 
and the $m$-plane $Q$ is invariant under $\Sigma$. 
Hence 
$$m = {\rm dim}(\Eta) = {\rm dim}(\Gamma/\Nu) = {\rm dim}(E^n/V).$$
By conjugating $\Gamma$ by a translation, we may assume that $Q$ is a vector subspace of $E^n$. 
Let $a_1+A_1,\ldots, a_m+A_m$ be generators of $\Eta$. 
Then $a_i\in Q$ for each $i$ and $A_i$ fixes $Q$ pointwise for each $i$. 
Let $k$ be the order of the point group of $\Gamma$. 
Then $(a_i+A_i)^k = ka_i+I$ for each $i$. 
Hence, by replacing $\Eta$ by a subgroup of finite index, 
we may assume that $A_i=I$ for each $i$. 

Now $p(a_i+I)$ acts on $E^n/V$ by $\Nu(a_i+I)(V+x) = V+x+a_i$, 
and so $p(a_i+I)$ acts as a translation on $E^n/V$ for each $i$. 
As $p(\Eta)$ has finite index in $\Gamma/\Nu$, 
we have that $p(\Eta)$ has finite index in the translation subgroup of $\Gamma/\Nu$. 
Therefore the vectors $V+a_1,\ldots, V+a_m$ span $E^n/V$. 
Hence the quotient map $\pi:E^n\to E^n/V$ maps $Q$ isomorphically onto $E^n/V$. 
Therefore $V\cap Q = \{0\}$. 

If $x\in E^n$, then $x$ can be written uniquely as $x = x_V+x_Q$ 
with $x_V\in V$ and $x_Q\in Q$. 
Define $\phi:E^n/V\to Q$ by $\phi(V+x) = x_Q$. 
Then $\phi$ is a well-defined linear isomorphism. 

Let $a+A\in \Sigma$. 
Then $a\in Q$ and $A$ leaves both $V$ and $Q$ invariant. 
Observe that 
\begin{eqnarray*}
\phi(p(a+A)(V+x)) & = & \phi(\Nu(a+A)(V+x)) \\
& = & \phi(V+a+Ax) \\
& = & a + Ax_Q \\
& = & (a+A)x_Q \ \ = \ \ (a+A)\phi(V+x). 
\end{eqnarray*}
Hence $\phi$ induces an affine map $\ov\phi:(E^n/V)/(\Gamma/\Nu)\to E^n/\Gamma$ 
whose image is $\Gamma Q/\Gamma$. 
Observe that 
\begin{eqnarray*}
\eta\ov\phi((\Gamma/\Nu)(V+x)) & = & \eta(\Gamma\phi(V+x)) \\
& = & \eta(\Gamma x_Q) \\
& = & (\Gamma/\Nu)(V+x_Q) \ \ = \ \ (\Gamma/\Nu)(V+x). 
\end{eqnarray*}
Therefore $\eta\ov\phi$ is the identity map. 
Thus $\ov\phi$ is an affine section of $\eta$. 
\end{proof}

\begin{lemma} 
Let $\Nu$ be a complete normal subgroup of an $n$-dimensional space group $\Gamma$, 
and let $V = {\rm Span}(\Nu)$. 
Let $\eta: E^n/\Gamma \to (E^n/V)/(\Gamma/\Nu)$ be the fibration projection determined by $\Nu$. 
If $\eta$ has an affine section $\sigma:(E^n/V)/(\Gamma/\Nu)\to E^n/\Gamma$ such that 
${\rm Im}(\sigma)$ intersects a fiber $F_0$ of $\eta$ at an ordinary point $x_0$ of $F_0$, 
then the space group extension 
$ 1 \to {\rm N} \to \Gamma \to \Gamma/\Nu\to 1$
splits. 
\end{lemma}
\begin{proof}
By conjugating $\Gamma$ by a translation, 
we may assume that $x_0 = \Gamma 0$. 
Then $\sigma$ lifts to an affine map 
$\tilde\sigma:E^n/V\to E^n$ such that 
${\rm Im}(\tilde\sigma)$ is a vector subspace $Q$ of $E^n$ 
and the following diagram commutes
\[\begin{array}{ccc}
E^n/V & {\buildrel \tilde{\sigma} \over \longrightarrow} & E^n \vspace{.1in}\\ 
\pi_{\Gamma/\Nu}\ \downarrow \hspace{.2in} & & \phantom{\pi_\Gamma}\downarrow\ \pi_\Gamma \vspace{.1in} \\
(E^n/V)/(\Gamma/\Nu) & {\buildrel \sigma \over \longrightarrow} &  E^n/\Gamma
  \end{array}
\] 
where the vertical maps are the quotient maps. 
Now we have 
$$\pi_\Gamma(Q) = \pi_\Gamma(\tilde\sigma(E^n/V)) = 
\sigma(\pi_{\Gamma/\Nu}(E^n/V)) = {\rm Im}(\sigma).$$
Let $\tilde\eta:E^n\to E^n/V$ be the quotient map. 
Then $\pi_{\Gamma/\Nu}\tilde\eta=\eta\pi_\Gamma$, and so 
$$\pi_{\Gamma/\Nu}\tilde\eta(Q) = \eta\pi_\Gamma(Q) = 
\eta({\rm Im}(\sigma)) = (E^n/V)/(\Gamma/\Nu).$$
As $\tilde\eta(Q)$ is a vector subspace of $E^n/V$, 
we deduce that $\tilde\eta(Q) =E^n/V$. 
Therefore $\tilde\eta\tilde\sigma: E^n/V\to E^n/V$ 
is an affine homeomorphism, 
and so $\tilde\sigma:E^n/V\to E^n$ is an affine embedding 
whose image is $Q$. 

Let $\Sigma$ be the stabilizer of $Q$ in $\Gamma$, 
and let $a+A\in\Nu\cap\Sigma$. 
Then $a\in V\cap Q =\{0\}$. 
Hence $A = I$, since $\Gamma 0$ is an ordinary point of 
the fiber $F_0 = \pi_\Gamma(V)$, 
which is isometric to $V/\Gamma_V$. 
Therefore $N\cap \Sigma = \{I\}$. 

Suppose $a+A\in\Sigma$ and $a+A$ fixes $Q$ pointwise. 
Then $a=0$, since $0\in Q$. 
Let $x \in Q$. 
Write $x = v+w$ 
with $v\in V$ and $w\in V^\perp$. 
Then $Ax = Av+Aw$, 
and so $x = Av+Aw$. 
Now $Av\in V$ and $Aw\in V^\perp$, 
and so $Av = v$ and $Aw=w$. 
Hence $A$ fixes $V^\perp$ pointwise, 
since $\tilde\eta(Q)=E^n/V$. 
Therefore $A\in \Nu$. 
Hence $A = I$, since $\Gamma 0$ 
is an ordinary point of $F_0$. 
Thus $\Sigma$ acts effectively on $Q$. 

Suppose $x\in Q$, and $\gamma\in \Gamma$, and $y=\gamma x\in Q$. 
Choose $r>0$ small enough so that 
$B(y,r)\cap B(\alpha y,r)=\emptyset$ unless $\alpha\in\Gamma_y$, 
the stabilizer of $y$ in $\Gamma$. 
Then we have 
$$\pi_\Gamma^{-1}\big(\pi_\Gamma(B(y,r)\cap Q)\big)\cap B(y,r) = 
\mathop{\cup}_{\alpha\in\Gamma_y}\alpha\big(B(y,r)\cap Q\big).$$
We have $\pi_\Gamma(B(y,r)\cap\gamma Q) = \pi_\Gamma(B(y,r)\cap Q)$. 
Therefore, we have 
$$B(y,r)\cap\gamma Q \ \subset \mathop{\cup}_{\alpha\in\Gamma_y}\big(B(y,r)\cap\alpha Q\big).$$
Hence, we have
$$B(y,r)\cap\gamma Q\ = \mathop{\cup}_{\alpha\in\Gamma_y}\big(B(y,r)\cap\alpha Q\cap\gamma Q\big).$$
Therefore $\gamma Q = \alpha Q$ for some $\alpha\in \Gamma_y$, since $\Gamma_y$ is finite. 
Hence $\alpha^{-1}\gamma Q = Q$ and $\alpha^{-1}\gamma x = y$. 
Thus $\alpha^{-1}\gamma \in \Sigma$ and $\pi_\Gamma$ induces an isometry 
from $Q/\Sigma$ to ${\rm Im}(\sigma)$. 

We have a commutative diagram 

\[\begin{array}{ccc}
Q \hspace{-.2in}& {\buildrel \tilde{\eta}_1 \over \longrightarrow} &\!\!\!\!\! E^n/V \vspace{.1in}\\ 
(\pi_\Gamma)_1 \downarrow \phantom{(\pi_\Gamma)_1} \hspace{-.2in}& & \phantom{\pi_{\Gamma/\Nu}}\downarrow\ \pi_{\Gamma/\Nu} \vspace{.1in} \\
{\rm Im}(\sigma) \hspace{-.2in}& {\buildrel \eta_1 \over \longrightarrow} &  (E^n/V)/(\Gamma/\Nu)
  \end{array}
\] 
with $\tilde\eta_1,\eta_1,(\pi_\Gamma)_1$ the restrictions of $\tilde\eta,\eta,\pi_\Gamma$, 
respectively; moreover, $\tilde\eta_1$ and $\eta_1$ are homeomorphisms 
and the fibers of $(\pi_\Gamma)_1$ are the orbits of the action of $\Sigma$ on $Q$. 
Let $x\in Q$ such that $\pi_{\Gamma/\Nu}\tilde\eta_1(x) = \pi_{\Gamma/\Nu}(V+x)$ 
is an ordinary point of $(E^n/V)/(\Gamma/\Nu)$, 
and let $\Nu\gamma\in \Gamma/\Nu$. 
Then there exists $\gamma'\in\Sigma$ such that 
$\tilde\eta_1(\gamma'x) = (\Nu\gamma)\tilde\eta_1(x)$. 
Hence $(\Nu\gamma')(V+x) = (\Nu\gamma)(V+x)$, 
and so $\Nu\gamma'=\Nu\gamma$, 
since $\pi_{\Gamma/\Nu}(V+x)$ is an ordinary point of $(E^n/V)/(\Gamma/\Nu)$. 
Therefore $\Nu\Sigma =\Gamma$. 
Hence, the space group extension 
$$ 1 \to {\rm N}\ {\buildrel i\over \longrightarrow}\ \Gamma\ {\buildrel p\over\longrightarrow} 
\ \Gamma/\Nu\to 1$$
splits with $p$ mapping $\Sigma$ isometrically onto $\Gamma/\Nu$. 
\end{proof}

\begin{theorem} 
Let $\Nu$ be a complete, torsion-free, normal subgroup of an $n$-dimen\-sional space group $\Gamma$, 
and let $V = {\rm Span}(\Nu)$. 
Let $\eta: E^n/\Gamma \to (E^n/V)/(\Gamma/\Nu)$ be the fibration projection determined by $\Nu$. 
Then $\eta$ has an affine section if and only if the space group extension 
$1 \to {\rm N} \to \Gamma \to \Gamma/\Nu\to 1$
splits. 
\end{theorem}
\begin{proof}
If the space group extension $1 \to {\rm N} \to \Gamma \to \Gamma/\Nu\to 1$ splits, 
then $\eta$ has an affine section by Lemma 6. 
If $\eta$ has an affine section $\sigma$, 
then ${\rm Im}(\sigma)$ intersects every generic fiber of $\eta$ at an ordinary point, 
since every point of a generic fiber is an ordinary point, because $\Nu$ is torsion-free. 
Therefore, the space group extension $1 \to {\rm N} \to \Gamma \to \Gamma/\Nu\to 1$ splits 
by Lemma 7. 
\end{proof}

\begin{lemma} 
Let $\Nu$ be a complete normal subgroup of an $n$-dimensional space group $\Gamma$, 
and let $V = {\rm Span}(\Nu)$. 
Let $\eta: E^n/\Gamma \to (E^n/V)/(\Gamma/\Nu)$ be the fibration projection determined by $\Nu$. 
If $V/\Nu$ has a point $\Nu v_0$ that is fixed by every isometry of $V/\Nu$,  
then the map $\sigma: (E^n/V)/(\Gamma/\Nu) \to E^n/\Gamma$, defined by 
$\sigma((\Gamma/\Nu)(V+x)) = \Gamma(v_0+x)$ for each $x\in V^\perp$, 
is a section of $\eta$ and an affine isometric embedding. 
\end{lemma}
\begin{proof}
Let $b+B\in\Gamma$. 
Write $b= c+d$ with $c\in V$ and $d\in V^\perp$. 
Then $(b+B)V = V+d$. 
Let $a+A\in\Nu$ and let $v\in V$. 
Then we have 
$$(b+B)((a+A)v) = (b+B)(a+A)(b+B)^{-1}(b+B)v= (a'+A')((b+B)v)$$
with $a'+A'\in \Nu$, because $\Nu$ is a normal subgroup of $\Gamma$. 
Hence $b+B$ induces an isometry $(b+B)_\ast$ from $V/\Nu$ to $(V+d)/\Nu$ 
defined by $(b+B)_\ast(\Nu v) = \Nu(b+Bv)$. 
Now $(b+B)_\ast(\Nu v_0) = \Nu(b+Bv_0)$, 
and so $\Nu(b+Bv_0)$ is fixed by every isometry of $(V+d)/\Nu$. 

Next, we have $(d+I)V = V+d$ and $(d+I)((a+A)v) = (a+A)((d+I)v)$. 
Hence $d+I$ induces an isometry $(d+I)_\ast$ from $V/\Nu$ to $(V+d)/\Nu$ 
defined by $(d+I)_\ast(\Nu v) = \Nu(v+d)$. 
Now $(d+I)_\ast(\Nu v_0) = \Nu(v_0+d)$. 
Hence we have 
$$(d+I)_\ast(b+B)_\ast^{-1}(\Nu(b+Bv_0)) = \Nu(v_0+d),$$
and so $\Nu(v_0+d) = \Nu(b+Bv_0)$. 
Therefore there is an $a+A\in\Nu$ such that $(a+A)(b+Bv_0) = v_0+d$. 
Hence $a+Ac+ABv_0 = v_0$. 

The map $\tilde\sigma:E^n/V\to E^n$, defined by $\tilde\sigma(V+x) = v_0+x$ 
for each $x\in V^\perp$, is an affine isometric embedding. 
Observe that 
\begin{eqnarray*}
\tilde\sigma((b+B)(V+x)) 
& = & \tilde\sigma(V+d+Bx) \\
& = & v_0+d+Bx \\
& = & a+Ac+ABv_0+d+Bx \\
& = & a+Ac+ABv_0+Ad+ABx \\
& = & (a+Ab+AB)(v_0+x) \ \ 
 = \ \ (a+A)(b+B)\tilde\sigma(V+x).
\end{eqnarray*}
Hence $\tilde\sigma$ induces a map $\sigma:(E^n/V)/(\Gamma/\Nu)\to E^n/\Gamma$ 
defined by $\sigma((\Gamma/\Nu)(V+x)) = \Gamma(v_0+x)$ 
for each $x\in V^\perp$. 
Now we have 
$$\eta\sigma((\Gamma/\Nu)(V+x)) = \eta(\Gamma(v_0+x)) = (\Gamma/\Nu)(V+x),$$
and so $\sigma$ is a section of $\eta$ and an affine isometric embedding. 
\end{proof}

\begin{theorem} 
Let $\Nu$ be a complete normal subgroup of an $n$-dimensional space group $\Gamma$, 
and let $V = {\rm Span}(\Nu)$. 
If $V/\Nu$ has an ordinary point $\Nu v_0$ that is fixed by every isometry of $V/\Nu$,  
then the space group extension $1\to\Nu\to \Gamma\to \Gamma/\Nu\to 1$ splits. 
\end{theorem}
\begin{proof}
By conjugating $\Gamma$, we may assume that $\Gamma V/\Gamma$ is a generic fiber of $\eta$. 
Then $\Gamma V/\Gamma$ is isometric to $V/\Nu$. 
By Lemma 8, the fibration projection $\eta: E^n/\Gamma\to (E^n/V)/(\Gamma/\Nu)$ 
has an affine section $\sigma$ such that ${\rm Im}(\sigma)$ intersects the fiber $\Gamma V/\Gamma$ 
in the ordinary point $\Gamma v_0$. 
Hence the space group extension $1\to\Nu\to \Gamma\to \Gamma/\Nu\to 1$ splits by Lemma 7. 
\end{proof}

For example, consider a space group extension $1\to\Nu\to\Gamma\to\Gamma/\Nu\to1$ 
such that $\Nu$ is an infinite dihedral group. Then $V/\Nu$ is a closed interval. 
The midpoint of the closed interval $V/\Nu$ is an ordinary point of $V/\Nu$ 
that is fixed by the nonidentity isometry of $V/\Nu$. 
Hence the space group extension $1\to\Nu\to\Gamma\to\Gamma/\Nu\to1$ splits by Theorem 18. 

We next consider an example that shows that the hypothesis that $\Nu v_0$ is an ordinary point 
cannot be dropped in Theorem 18. 
Let $\Gamma$ be the 3-dimensional space group with IT number 113 in Table 1B of \cite{B-Z}. 
Then $\Gamma = \langle t_1,t_2,t_3,\alpha,\beta,\gamma\rangle$ 
where $t_i = e_i+I$ for $i=1,2,3$ are the standard translations, 
and $\alpha = \frac{1}{2}e_1+\frac{1}{2}e_2+A$, 
$\beta=\frac{1}{2}e_1+B$, $\gamma =\frac{1}{2}e_2+C$, and 
$$A = \left(\begin{array}{rrr} -1 & 0 & 0\\ 0 & -1 & 0 \\ 0 & 0 & 1  \end{array}\right),\ \ 
B = \left(\begin{array}{rrr} 0 & 1 & 0  \\ -1 & 0 & 0   \\ 0 & 0 & -1 \end{array}\right), \ \ 
C = \left(\begin{array}{rrr} -1 & 0 & 0  \\ 0 & 1 & 0   \\ 0 & 0 & -1 \end{array}\right).$$
The group $\Nu = \langle t_1,t_2,\alpha,\beta\gamma\rangle$ is a complete normal subgroup 
of $\Gamma$ with $V = {\rm Span}(\Nu) = {\rm Span}\{e_1,e_2\}$. 
The isomorphism type of $\Nu$ is $2\!\ast\!22$ in Conway's notation \cite{Conway} 
or $cmm$ in the international notation \cite{S}. 
The flat orbifold $V/\Nu$ is a pointed hood.  
The orbifold $V/\Nu$ has a unique cone point $\Nu(\frac{1}{4}e_1+\frac{1}{4}e_2)$ 
which corresponds to the fixed point $\frac{1}{4}e_1+\frac{1}{4}e_2$ of the halfturn $\alpha$. 
Hence the cone point of $V/\Nu$ is fixed by every isometry of $V/\Nu$. 
Therefore the fibration projection $\eta:E^3/\Gamma\to(E^3/V)/(\Gamma/\Nu)$ 
has an isometric section by Lemma 8. 
However the space group extension $1\to \Nu\to \Gamma\to \Gamma/\Nu\to 1$ 
does not split, since $\gamma$ projects to an element of order 2 in $\Gamma/\Nu$, 
but $((d+D)\gamma)^2\neq I$ for all $d+D$ in $\Nu$. 
To see why, observe that there are only four possibilities for $D$, namely $D=I, A, BC, ABC$. 
Suppose $D=I$. Then $d=ke_1+\ell e_2$ for some integers $k$ and $\ell$, and we have
$$\big((d+D)\gamma\big)^2 = (1+2\ell)e_2+I \neq I.$$
The proofs of the other three cases for $D$ are similar. 
This example also shows that the hypothesis that $\Nu$ is torsion-free 
cannot be dropped in Theorem 17.

\section{Seifert Fibrations} 

We call a geometric fibration of a flat $n$-orbifold $M$ over a flat $(n-1)$-orbifold $B$ 
with generic fiber a connected, compact, flat 1-orbifold $F$ a {\it geometric Seifert fibration}. 
Here $F$ is either a circle or a closed interval. 
Seifert fibrations of compact flat 3-orbifolds have been 
studied by Bonahon and Siebenmann \cite{B-S}, Dunbar \cite{Dunbar}, 
and Conway et al. \cite{C-T}. 
Every Seifert fibration of a compact flat 2- or 3-orbifold 
is isotopic to a geometric Seifert fibration   
by Proposition 2.12 of Boileau, Maillot, and Porti \cite{B-M-P} 
and the discussion in \cite{B-S}. 

Let $F$ be a connected, compact, flat 1-orbifold, and let $B$ be a flat $(n-1)$-orbifold $B$. 
Then $F\times B$ is a flat $n$-orbifold. 
The natural projection of $F\times B$ onto $B$ is a geometric Seifert fibration 
over $B$ with generic fiber $F$.  

Let $\Iota$ be a closed interval,  
let $\hat B$ be a connected, complete, flat $(n-1)$-orbifold, 
and let $\sigma$ be an involution of $\Iota\times \hat B$ 
which acts diagonally, as a reflection on $\Iota$, and effectively and 
isometrically on $\hat B$. 
Let $M = (\Iota\times \hat B)/\langle \sigma\rangle$. 
Then $M$ is a flat $n$-orbifold and $B = \hat B/\langle \sigma\rangle$ 
is a flat $(n-1)$-orbifold. 
The natural projection of $M$ onto $B$ 
is a geometric Seifert fibration over $B$ 
with generic fiber $\Iota$. The flat orbifold $M$ is called 
the {\it twisted $\Iota$-bundle} over $B$ 
determined by the orbifold double cover $\hat B$ of $B$. 

\begin{theorem} 
Let $M$ be a connected, complete, flat $n$-orbifold, let $B$ be a flat $(n-1)$-orbifold, 
and let $\eta:M\to B$ be a geometric Seifert fibration with generic fiber a closed interval $\Iota$. 
Let $\dot{\Iota}$ be the set of endpoints of $\Iota$, 
let $\dot{M}$ be the the union of all the endpoints of the fibers of $\eta$ 
that are determined by the endpoints of $\Iota$, 
and let $\dot{\eta}:\dot{M} \to B$ be the restriction of $\eta$. 
Then $\dot{M}$ is a complete, flat $(n-1)$-orbifold, and  
$\dot{\eta}$ is a geometric fibration over $B$ with generic fiber 
the flat $0$-orbifold $\dot{\Iota}$. 
If $\dot{M}$ is disconnected, then $\eta$ is isometrically equivalent to the natural 
projection $\Iota \times B \to B$. 
If $\dot{M}$ is connected, then $\eta$ is isometrically equivalent 
to the natural projection of the twisted $\Iota$-bundle over $B$ determined 
by the orbifold double covering $\dot{\eta}: \dot{M} \to B$. 
\end{theorem}
\begin{proof}
That $\dot{M}$ is a flat $(n-1)$-orbifold and $\dot{\eta}$ is a geometric fibration over $B$ 
with generic fiber $\dot{\Iota}$ follows from the definition of the geometric fibration $\eta$. 
The orbifold $\dot{M}$ is complete,  
since $\dot{M}$ is a closed subspace of the complete metric space $M$. 
Let $B_{1/2}$ be the union of all the points of the fibers of $\eta$ 
that are determined by the midpoint of $\Iota$. 
Then $B_{1/2}$ is a flat $(n-1)$-orbifold, and $\eta$ maps $B_{1/2}$ isometrically onto $B$. 
The geometric fibration $\dot{\eta}:\dot{M}\to B$ is geometrically equivalent 
to the fiberwise projection from $\dot{M}$ to $B_{1/2}$.  

Suppose $\dot{M}$ is disconnected. 
Then $\dot{M}$ has exactly two connected components $B_0$ and $B_1$, 
and $\dot{\eta}$ maps $B_i$ isometrically onto $B$ for each $i$, 
since $\dot{\eta}:\dot{M} \to B$ is an orbifold double covering and $B$ is connected. 
Hence $B_{1/2}$ is two-sided in $M$. 
Parameterize $\Iota$ so that $\Iota = [0,\ell]$. 
Define $\phi: \Iota\times B \to M$ by $\phi(t,y) = y_t$ 
where $y_t$ is the point in $\eta^{-1}(y)$ 
which is at a distance $t$ from $B_0$ along $\eta^{-1}(y)$ towards $B_1$.  
Then $\phi$ is an isometry, and $\eta$ is isometrically equivalent 
to the natural projection $\Iota\times B\to B$. 

Suppose $\dot{M}$ is connected. 
Then $B_{1/2}$ is one-sided in $M$. 
Let $\tau:\dot{M}\to \dot{M}$ be the continuous involution 
that transposes the endpoints of the fibers of $\eta$ that are isometric to $\Iota$. 
Then $\tau$ is an isometry. 
Note that $x$ is fixed by $\tau$ if and only if $\eta^{-1}(\eta(x))$ 
is a singular fiber isometric to $\Iota$ folded in half. 
The geometric fibration $\dot{\eta}: \dot{M}\to B$ 
induces an isometry from $\dot{M}/\langle \tau\rangle$ to $B$.  
Define $\phi: \Iota\times \dot{M} \to M$ 
by $\phi(t,x) = x_t$ where $x_t$ is the point of $\eta^{-1}(\eta(x))$ 
which is at a distance $t$ from $x$ if $t\leq \ell/2$ or at a distance $\ell-t$ from $\tau(x)$ 
if $t\geq \ell/2$. 
Let $\sigma: \Iota\times \dot{M}$ be the isometry 
that acts diagonally, as the reflection in $\Iota$, and by $\tau$ on $\dot{M}$. 
Then $\phi$ induces an isometry from $(\Iota\times \dot{M})/\langle\sigma\rangle$ to $M$, 
and $\eta$ is isometrically equivalent 
to the natural projection of the twisted $\Iota$-bundle over $B$ determined 
by the orbifold double covering $\dot{\eta}: \dot{M} \to B$. 
\end{proof}

\begin{theorem} 
If $1 \to {\rm N}\ {\buildrel i\over \longrightarrow}\ \Gamma\ {\buildrel p\over\longrightarrow} 
\ \Gamma/\Nu\to 1$ is a space group extension such that $\Nu$ is an infinite dihedral group,  
then $\Gamma$ has a subgroup $\Sigma$ such that $p$ maps $\Sigma$ isomorphically onto $\Gamma/\Nu$,  
and either $\Gamma = \Nu\times\Sigma$ and $\Sigma$ is unique and orthogonal to $\Nu$,  
or else $\Sigma$ has a subgroup $\Sigma_0$ of index 2 
such that if $\Gamma_0 = \Nu\Sigma_0$, then $\Gamma_0=\Nu\times \Sigma_0$ 
and $\Sigma_0$ is unique, but $\Sigma$ is not necessarily unique; 
moreover $\Sigma_0$ is a complete normal subgroup of $\Gamma$, which is orthogonal to $\Nu$, 
and $\Gamma/\Sigma_0$ is an infinite dihedral group.  
\end{theorem}
\begin{proof}
Let $V = {\rm Span}(\Nu)$, and let $n={\rm dim}(\Gamma)$. 
Then ${\rm I} = V/\Nu$ is a closed interval and the fibration projection 
$\eta:E^n/\Gamma\to(E^n/V)/(\Gamma/\Nu)$ is a geometric Seifert fibration 
with generic fiber ${\rm I}$. 
Let $M=E^n/\Gamma$.  By Theorem 18, the group $\Gamma$ has a subgroup $\Sigma$ 
such that $p$ maps $\Sigma$ isomorphically onto $\Gamma/\Nu$, 
and by Theorem 19, either $\Gamma = \Nu\times\Sigma$, with $\Sigma$ orthogonal to $\Nu$,  
if $\dot{M}$ is disconnected, 
or else $\Sigma$ has a subgroup $\Sigma_0$ of index 2, if $\dot{M}$ is connected, 
corresponding to the orbifold double cover $\dot{M}$ of $(E^n/V)/(\Gamma/\Nu)$, 
such that if $\Gamma_0 = \Nu\Sigma_0$, then $\Gamma_0=\Nu \times \Sigma_0$, 
since ${\rm I}\times \dot{M}$ double covers $M$. 
If $\Gamma = \Nu\times\Sigma$, then $\Sigma$ is unique, since $\Sigma$ 
is the centralizer of $\Nu$ in $\Gamma$. 
If $\Gamma \neq \Nu\times\Sigma$, then $\Sigma_0$ is unique, 
since $\Sigma_0$ is the centralizer of $\Nu$ in $\Gamma$. 
To see that $\Sigma$ is not necessarily unique, 
see example (5) in the next section. 
The group $\Sigma_0$ is normal in $\Gamma$, since $\Sigma_0$ is 
the centralizer of a normal subgroup of $\Gamma$. 
The group $\Sigma_0$ is complete and $\Gamma/\Sigma_0$ 
is an infinite dihedral group, since $\Sigma_0$ corresponds 
to the generic fiber of the geometric fibration of $({\rm I}\times \dot{M})/\langle \sigma\rangle$ 
over ${\rm I}/\langle \sigma\rangle$ induced by projection along the first factor. 
Finally $\Sigma_0$ is orthogonal to $\Nu$, since $\Sigma$ can be chosen to be orthogonal to $\Nu$ 
as in the proof of Lemma 8. 
\end{proof}

\section{Reducible 2-Dimensional Crystallographic Groups} 

Let $\Gamma$ be a 2-dimensional space group. 
Then every nontrivial geometric fibration of $E^2/\Gamma$ 
is a Seifert fibration. 
If $b+B\in \Gamma$ and $B$ has no 1-dimensional invariant 
vector space, then $E^2/\Gamma$ does not Seifert fiber. 
Hence if $B$ is a rotation of order 3 or 4, 
then $E^2/\Gamma$ does not Seifert fiber. 
This excludes 8 of the 2-dimensional 
space group isomorphism types. 
The orbifolds of the remaining 9 isomorphism types do Seifert fiber. 
We now describe all of these geometric Seifert fibrations. 
We denote a circle by $\circle$ and a closed interval by $\Iota$. 

We consider the groups in the order of Table 1A of Brown et al. \cite{B-Z} 
which corresponds to the IT order \cite{IT}. 
We use the generators for the representatives of the isomorphism types 
of the 2-dimensional space groups 
listed in Table 1A of \cite{B-Z}.  
Let $t_1 = e_1 + I$ and $t_2 = e_2 + I$ be the standard basis translations, and let  
$$A = \left(\begin{array}{rr} -1 & 0 \\ 0 & -1 \end{array}\right),\ \ 
B = \left(\begin{array}{rr} 1 & 0 \\ 0 & -1 \end{array}\right), \ \ 
C = \left(\begin{array}{rr} 0 & 1 \\ 1 & 0 \end{array}\right).$$
Let $a,b$ be relatively prime integers, let $c,d$ be integers such that $ad-bc=1$,   
and let $\phi:E^2\to E^2$ be the linear automorphism defined by $\phi(e_1) = (a,b)$ 
and $\phi(e_2) = (c,d)$. 
The symbol $\rtimes$ denotes a semidirect product of groups.

(1) Let $\Gamma=\langle t_1,t_2\rangle$. 
The isomorphism type of $\Gamma$ is $\circ\, (p1)$. 
Here $\circ$ is Conway's notation \cite{Conway} and $p1$ is 
the IT notation \cite{S}. 
The orbifold $E^2/\Gamma$ is a flat torus 
and $E^2/\Gamma$ is the cartesian product $\circle\times \circle$. 
The proper, complete, normal subgroups of $\Gamma$ 
are of the form $\langle t_1^at_2^b\rangle$. 
We have $\Gamma= \langle t_1^at_2^b\rangle \times \langle t_1^ct_2^d\rangle$, 
and so $E^2/\Gamma$ is a geometric trivial fiber bundle 
over $\circle$, with fiber $\circle$, in infinitely many ways. 
The linear automorphism $\phi$ normalizes $\Gamma$ and 
conjugates $\langle t_1\rangle$ to $\langle t_1^at_2^b\rangle$. 
Therefore, all the geometric Seifert fiberings of $E^2/\Gamma$ are affinely equivalent. 

(2) Let $\Gamma=\langle t_1,t_2, A\rangle$. 
The isomorphism type of $\Gamma$ is $2222\,(p2)$ and $E^2/\Gamma$ is a flat pillow. 
The proper, complete, normal subgroups of $\Gamma$ 
are of the form $\langle t_1^at_2^b\rangle$. 
We have $\Gamma= \langle t_1^at_2^b\rangle \rtimes \langle t_1^ct_2^d, A\rangle$, 
and so $E^2/\Gamma$ geometrically fibers over $\Iota$,  
with generic fiber $\circle$, in infinitely many ways. 
The linear automorphism $\phi$ normalizes $\Gamma$ and 
conjugates $\langle t_1\rangle$ to $\langle t_1^at_2^b\rangle$. 
Therefore, all the geometric Seifert fiberings of $E^2/\Gamma$ are affinely equivalent.

(3) Let $\Gamma=\langle t_1,t_2, B\rangle$. 
The isomorphism type of $\Gamma$ is $*\!*(pm)$ and $E^2/\Gamma$ is a flat annulus. 
The proper, complete, normal subgroups of $\Gamma$ 
are $Z(\Gamma)=\langle t_1\rangle$ and $\langle t_2,B\rangle$. 
We have $\Gamma = \langle t_1\rangle\times \langle t_2,B\rangle$, 
and so $E^2/\Gamma$ is the cartesian product $\circle\times \Iota$.

(4) Let $\Gamma=\langle t_1,t_2,\beta\rangle$ where $\beta = \frac{1}{2}e_1+B$. 
The isomorphism type of $\Gamma$ is $\times\!\times (pg)$ and $E^2/\Gamma$ is a flat Klein bottle.  
The proper, complete, normal subgroups of $\Gamma$  
are $Z(\Gamma)=\langle t_1\rangle$ and $\langle t_2\rangle$. 
Now $\Gamma/\langle t_1\rangle \cong D_\infty$, 
and so $E^2/\Gamma$ geometrically fibers over $\Iota$ with generic fiber $\circle$. 
The extension $\langle t_1\rangle \to \Gamma \to D_\infty$ 
does not split, since $\Gamma$ is torsion-free. 
Also $\Gamma = \langle t_2\rangle\rtimes\langle \beta\rangle$, 
and so $E^2/\Gamma$ is a geometric fiber bundle over $\circle$ with fiber $\circle$. 

(5) Let $\Gamma=\langle t_1,t_2,C\rangle$. 
The isomorphism type of $\Gamma$ is $*\!\times (cm)$ and $E^2/\Gamma$ is a flat M\"obius band.  
The proper, complete, normal subgroups of $\Gamma$  
are $Z(\Gamma)=\langle t_1t_2\rangle$ and $\langle t_1t_2^{-1},C\rangle$. 
Now $\Gamma/\langle t_1t_2\rangle \cong D_\infty$, 
and so $E^2/\Gamma$ geometrically fibers over $\Iota$ with generic fiber $\circle$. 
The extension $\langle t_1t_2\rangle \to \Gamma \to D_\infty$ does not split, 
since $\Gamma$ is not the direct product of $\langle t_1t_2\rangle$ and $D_\infty$. 
Also $\Gamma = \langle t_1t_2^{-1},C\rangle\rtimes\langle t_iC\rangle$ for $i=1,2$,  
and so $E^2/\Gamma$ is a geometric fiber bundle over $\circle$ with fiber $\Iota$. 
We have $(t_iC)^2 = t_1t_2$ for $i=1,2$, 
and $\langle t_1t_2\rangle$ is the centralizer of $\langle t_1t_2^{-1},C\rangle$ in $\Gamma$.  
As $(t_1C)^{-1} = t_2^{-1}C$, we have that $\langle t_1C\rangle \neq \langle t_2C\rangle$.  
Hence the subgroup $\Sigma = \langle t_1C\rangle$ in Theorem 20 is not necessarily unique.

(6) Let $\Gamma=\langle t_1,t_2, A, B\rangle$. 
The isomorphism type of $\Gamma$ is $\ast 2222\,(pmm)$ and $E^2/\Gamma$ is a square. 
The proper, complete, normal subgroups of $\Gamma$ 
are $\langle t_1, AB\rangle$ and $\langle t_2,B\rangle$. 
Now $\Gamma = \langle t_1, AB\rangle\times \langle t_2,B\rangle$, 
and so $E^2/\Gamma$ is the cartesian product $\Iota\times \Iota$. 
The linear automorphism $\sigma: E^2\to E^2$, defined by 
$\sigma(e_1) =e_2$ and $\sigma(e_2) = e_1$,  
normalizes $\Gamma$ and conjugates $\langle t_1, AB\rangle$ to $\langle t_2, B\rangle$. 
Therefore, the two geometric Seifert fiberings of $E^2/\Gamma$ are isometrically equivalent.

(7) Let $\Gamma=\langle t_1,t_2,A,\beta\rangle$ where $\beta = \frac{1}{2}e_2+B$. 
The isomorphism type is $22\ast (pmg)$ and $E^2/\Gamma$ is a pillowcase.  
The proper, complete, normal subgroups of $\Gamma$  
are $\langle t_1\rangle$ and $\langle t_2,\beta\rangle$. 
Now $\Gamma =\langle t_1\rangle \rtimes \langle A,\beta\rangle$, 
and so $E^2/\Gamma$ geometrically fibers over $\Iota$ with generic fiber $\circle$. 
Also $\Gamma = \langle t_2,\beta\rangle\rtimes\langle t_1,A\rangle$, 
and so $E^2/\Gamma$ geometrically fibers over $\Iota$ with generic fiber $\Iota$.  

(8) Let $\Gamma=\langle t_1,t_2,A,\beta\rangle$ where $\beta = \frac{1}{2}e_1+\frac{1}{2}e_2+B$. 
The isomorphism type of $\Gamma$ is $22\!\times (pgg)$ and $E^2/\Gamma$ is a projective pillow.  
The proper, complete, normal subgroups of $\Gamma$ 
are $\langle t_1\rangle$ and $\langle t_2\rangle$. 
Now $\Gamma/\langle t_i\rangle \cong D_\infty$ for $i=1,2$, 
and so $E^2/\Gamma$ geometrically fibers over $\Iota$, with generic fiber $\circle$, in two ways.  
The extension $\langle t_1\rangle \to \Gamma \to D_\infty$ does not split, 
since $\beta$ projects to an element of order 2 but $t_1^k\beta$ is a glide-reflection 
for each integer $k$. 
The linear automorphism $\sigma: E^2\to E^2$, defined by 
$\sigma(e_1) =e_2$ and $\sigma(e_2) = e_1$,  
normalizes $\Gamma$ and conjugates $\langle t_1\rangle$ to $\langle t_2\rangle$. 
Therefore, the two geometric Seifert fiberings of $E^2/\Gamma$ are isometrically equivalent. 
Hence the extension $\langle t_2\rangle \to \Gamma \to D_\infty$ also does not split. 

(9) Let $\Gamma=\langle t_1,t_2,A,C\rangle$. 
The isomorphism type of $\Gamma$ is $2\!\ast\! 22\,(cmm)$ and $E^2/\Gamma$ \linebreak 
is a pointed hood.  
The proper, complete, normal subgroups of $\Gamma$ 
are $\langle t_1t_2,AC\rangle$ and $\langle t_1t_2^{-1},C\rangle$. 
Now $\Gamma =\langle t_1t_2, AC\rangle \rtimes \langle t_1A, C\rangle$ 
and $\Gamma =\langle t_1t_2^{-1}, C\rangle \rtimes \langle t_1A, AC\rangle$,  
and so $E^2/\Gamma$ geometrically fibers over $\Iota$,   
with generic fiber $\Iota$, in two ways.  
The linear automorphism $\sigma: E^2\to E^2$, defined by 
$\sigma(e_1) =e_1$ and $\sigma(e_2) = -e_2$,  
normalizes $\Gamma$ and conjugates $\langle t_1t_2,AC\rangle$ to $\langle t_1t_2^{-1},C\rangle$. 
Therefore, the two geometric Seifert fiberings of $E^2/\Gamma$ are isometrically equivalent. 

\section{Co-Seifert Fibrations}  

Let $M$ be a flat $n$-orbifold. 
We call a geometric fibration of $M$ 
over a connected, compact, flat 1-orbifold $B$,  
with generic fiber a flat $(n-1)$-orbifold $F$, 
a {\it geometric co-Seifert fibration}. 
Here $B$ is either a circle $\circle$ or a closed interval $\Iota$. 
The structure of a geometric co-Seifert fibration $\eta:M\to B$ 
tells you a lot about the geometry and topology of the orbifold $M$. 
If the base $B$ is a circle $\circle$, 
then $M$ is a geometric fiber bundle over $\circle$ with fiber $F$. 
Hence $M$ is a mapping torus over $F$ with monodromy an isometry of $F$. 
Moreover, the singular set $\Sigma$ of $M$ is a fiber bundle over $\circle$ 
with fiber the singular set $\Sigma_\ast$ of $F$.

For example, take $\Gamma$ to be as in (5) in \S 10. 
Then $M = E^2/\Gamma$ is a flat M\"obius band, 
$M$ is a geometric fiber bundle over $\circle$ with fiber $\Iota$, 
and $M$ is the mapping torus over $\Iota$ with monodromy the reflection of $\Iota$ 
about the midpoint of $\Iota$.  
The singular set $\Sigma$ of $M$ is the boundary of the M\"obius band,  
which is a nontrivial fiber bundle over $\circle$ 
with fiber the endpoints of $\Iota$. 

Now suppose the base $B$ is a closed interval $\Iota$ with endpoints $0$ and $1$. 
Let $\Iota^\circ$ be the open interval $(0,1)$. 
Then $\eta^{-1}(\Iota^\circ)$ is the cartesian product of $F$ and $\Iota^\circ$. 
Hence $M$ is obtained from the open tube $\eta^{-1}(\Iota^\circ)$ by adjoining 
the two singular fibers $F_0 = \eta^{-1}(0)$ and $F_1=\eta^{-1}(1)$ of $\eta$ 
to the ends of the tube, one to each end. 
Let $G_0$ and $G_1$ be the finite groups of order 2 
in the definition of the geometric fibration $\eta$ 
such that $F_i$ is isometric to $F/G_i$ for $i=1,2$. 
There are three cases to consider. 

The first case is when $G_0$ and $G_1$ both act trivially on $F$. 
Then $M$ is the cartesian product $F\times\Iota$. 
The singular set $\Sigma$ of $M$ is $F_0\cup (\Sigma_\ast\times \Iota)\cup F_1$ 
where $\Sigma_\ast$ is the singular set of $F$.  

For example, take $\Gamma$ to be as in (6) in \S 10. 
Then $M= E^2/\Gamma$ is a square and $M$ is the cartesian product $\Iota\times\Iota$.  
The singular set of $M$ is the boundary of the square, 
which is the union of the 4 closed intervals $F_0, \Sigma_\ast\times \Iota, F_1$. 

The second case is when one of $G_0$ and $G_1$ acts trivially on $F$ and the other does not,  
say $G_1$ acts trivially on $F$ and $G_0$ does not. 
Then $M$ is isometric to $(F\times [-1,1])/G_0$ where $G_0$ acts diagonally, 
isometrically on $F$, and orthogonally on $[-1,1]$ by reflecting $[-1,1]$ about $0$;  
in other words, $M$ is a twisted $\Iota$-bundle over $F_0$ 
with monodromy determined by the action of $G_0$ on $F$.   
The singular set $\Sigma$ of $M$ is $\Sigma_0\cup(\Sigma_\ast\times (0,1])\cup F_1$ 
where $\Sigma_0$ is the singular set of $F_0$. 

For example, take $\Gamma$ to be as in (7) in \S 10. 
Then $M= E^2/\Gamma$ is a pillowcase and $M$ geometrically fibers over $\Iota$ 
with generic fiber $\circle$ as in the second case. 
We have that $M$ is a twisted $\Iota$-bundle over $\Iota$ 
with monodromy determined by a reflection of $\circle$.  
The singular set of $M$ is the union of the boundary circle $F_1$ of $M$ 
together with the two cone points which form the singular set of the closed interval $F_0$. 

As another example, take $\Gamma$ to be as in (9) in \S 10. 
Then $M=E^2/\Gamma$ is a pointed hood and $M$ geometrically fibers over $\Iota$ 
with generic fiber $\Iota$ as in the second case. 
We have that $M$ is a twisted $\Iota$-bundle over $\Iota$ 
with monodromy determined by the reflection of $\Iota$. 
The singular set $\Sigma$ of $M$ is the boundary of $M$ together with a single cone point.  
Observe that $\Sigma$ is the union of the closed interval $F_1$,  
two half-open intervals $\Sigma_\ast\times (0,1]$, and  
the end points of $F_0$, one of which is the cone point and the other joins 
the two half-open intervals $\Sigma_\ast\times (0,1]$ to form a closed interval.  

The third case is when $G_0$ and $G_1$ both act nontrivially on $F$. 
Let $M_0 = \eta^{-1}([0,1/2])$ and $M_1=\eta^{-1}([1/2,1])$, and let $F_{1/2} = \eta^{-1}(1/2)$. 
Then $M_i$ is a twisted $\Iota$-bundle over $F_i$ with monodromy determined by the action 
of $G_i$ on $F$ for $i=1,2$. 
Note that $M_i$ is a flat $n$-orbifold with totally geodesic boundary $F_{1/2}$ 
for each $i=1,2$. 
We have that $M = M_0\cup M_1$ with $M_0\cap M_1 = F_{1/2}$. 
The singular set $\Sigma$ of $M$ is $\Sigma_0\cup(\Sigma_\ast\times (0,1))\cup\Sigma_1$ 
where $\Sigma_i$ is the singular set of $F_i$ for $i=1,2$. 

For example, take $\Gamma$ to be as in (2) in \S 10. 
Then $M= E^2/\Gamma$ is a pillow and $M$ geometrically fibers over $\Iota$ 
with generic fiber $\circle$ as in the third case. 
Observe that $M_i$ is a pillowcase for each $i =1,2$, with $F_{1/2}$ their common boundary circle. 
The singular set of $M$ is four cone points which form the union of the endpoints  
of the closed intervals $F_0$ and $F_1$. 

As another example, take $\Gamma$ to be as in (8) in \S 10. 
Then $M= E^2/\Gamma$ is a projective pillow and $M$ geometrically fibers over $\Iota$ 
with generic fiber $\circle$ as in the third case. 
Here $G_0$ acts be a reflection on $\circle$ and $G_1$ acts by the a half-turn on $\circle$. 
Therefore $M_0$ is a pillowcase and $M_1$ is a M\"obius band,  
with $F_{1/2}$ is their common boundary circle. 
Therefore $M$ is topologically a projective plane. 
The singular set of $M$ consists of two cone points which are the endpoints 
of the closed interval $F_0$. 

As another example, take $\Gamma$ to be as in (7) in \S 10. 
Then $M= E^2/\Gamma$ is a pillowcase and $M$ geometrically fibers over $\Iota$ 
with generic fiber $\Iota$ as in the second case. 
Here $G_i$ acts on $\Iota$ by reflection for $i=1,2$,  
$M_i$ is a pointed hood for each $i=1,2$, and $F_{1/2}$ is a closed interval 
which is half the boundary of each hood. 
The singular set consists of the endpoints of the closed intervals $F_0$ and $F_1$ 
and two open intervals $\Sigma_\ast\times (0,1)$. 
The two open intervals $\Sigma_\ast\times (0,1)$ together with one endpoint from each of $F_0$ 
and $F_1$ form the boundary circle of $M$. 

Let $\Gamma$ be a  3-dimensional, orientation preserving, space group. 
In \cite{Dunbar}, Dunbar describes the topology and geometry of the flat orbifold $E^3/\Gamma$. 
When $E^3/\Gamma$ co-Seifert fibers, the structure of a co-Seifert fibration  
can effectively be used to determine the topology and geometry of the orbifold $E^3/\Gamma$.

\section{Reducible 3-Dimensional Crystallographic Groups}  

Let $\Gamma$ be a 3-dimensional space group 
with point group $\Pi$.  
Suppose $E^3/\Gamma$ geometrically fibers over 
a 1- or 2-dimensional flat orbifold. 
By Theorem 7, the geometric fibration of $E^3/\Gamma$ is determined  
by a complete normal subgroup ${\rm N}$ of $\Gamma$. 
Let $V = {\rm Span}({\rm N})$. 
Then $V$ is either a 1- or 2-dimensional vector subspace of $E^3$. 
As both $V$ and $V^\perp$ are invariant by $\Pi$, 
the group $\Pi$ has a 1-dimensional invariant vector space. 

The group $\Gamma$ is said to be {\it reducible} or {\it irreducible} 
according as its point group $\Pi$ has or has not 
a 1-dimensional invariant vector space. 
It is well known to crystallographers that 
reducibility of 3-dimensional space groups 
is equivalent to $\mathbb Z$-reducibility. 
Hence $E^3/\Gamma$ geometrically fibers  
over a flat 1- or 2-orbifold if and only if $\Gamma$ 
is reducible by Theorems 4, 7, and 11. 

There are six families of 3-dimensional space group isomorphism types,  
the triclinic, monoclinic, orthorhombic, tetragonal, hexagonal, and cubic families. 
See Table 1B of Brown et al.~\cite{B-Z} where the six families are listed in the above order. 
The cubic family consists of all the irreducible isomorphism types. 
There are 35 irreducible 3-dimensional 
space group isomorphism types 
and 184 reducible 3-dimensional space group isomorphism types.

All the geometric Seifert fibrations of $E^3/\Gamma$, up to affine equivalence,   
were neatly described by Conway et al.~in their 2001 paper \cite{C-T}.  
We will describe the Conway et al.~\ formulation of a geometric fibration of $E^3/\Gamma$ 
over a flat 2-orbifold and show how it determines a dual geometric fibration of $E^3/\Gamma$ 
over a flat 1-orbifold with the same invariant 1-dimensional vector space. 

Let $\Gamma$ is a reducible 3-dimensional space group. 
Conway et al.~represent a geometric fibration of $E^3/\Gamma$ over a flat 2-orbifold 
by an exact sequence
$$ 1 \to {\rm K} \longrightarrow \Gamma\ {\buildrel \pi\over\longrightarrow}\ {\rm H} \to 1$$
where ${\rm H}$ is a 2-dimensional space group 
and ${\rm K} = {\rm ker}(\pi)$. 
Here an element $g$ of $\Gamma$ is represented in the form $(g_H,g_V)$ 
where $\pi(g) = g_H$ and $g_V$ is the 3rd coordinate function of $g$. 
The group ${\rm K}$ is generated by either the translation $e_3+I$ 
or $e_3+I$ and the reflection $A={\rm diag}(1,1,-1)$.  
The group extension is specified by normalized choices of 
the 3rd coordinate functions for a set of generators of ${\rm H}$. 
The choices are such that $\{(I,g_V):\,g\in\Gamma\}$ 
is a discrete group $\Lambda$ of isometries of $E^3$ 
containing ${\rm K}$ as a normal subgroup of finite index. 
The projection $\phi: g\mapsto g_V$ is a group homomorphism 
whose image is a 1-dimensional space group isomorphic to $\Lambda$. 
Let ${\rm N} = {\rm ker}(\phi)$. 
Then ${\rm N}$ is complete by Theorem 5. 
Hence $E^3/\Gamma$ geometrically fibers over a flat 1-orbifold 
corresponding to the group extension
$$ 1 \to {\rm N} \longrightarrow \Gamma\ {\buildrel \phi\over\longrightarrow}\ \Lambda \to 1$$
by Theorem 4 with ${\rm Span}(\Nu) = {\rm Span}\{e_1,e_2\}$. 
Hence $\Kappa$ and $\Nu$ are orthogonal. 
We call $\Nu$ the {\it orthogonal dual} of $\Kappa$, 
and we say that $\Kappa$ and $\Nu$ correspond under {\it orthogonal duality}. 

As ${\rm K}\cap {\rm N} = \{I\}$, we have that $\pi$ maps ${\rm N}$ 
isometrically onto a normal subgroup of ${\rm H}$, moreover
\begin{eqnarray*}
{\rm H}/\pi({\rm N}) & \cong & (\Gamma/{\rm K})/({\rm KN}/{\rm K}) \\
                     & \cong & \Gamma/{\rm KN} \\
                     & \cong & (\Gamma/{\rm N})/({\rm KN}/{\rm N}) \ \ \cong \ \ \Lambda/{\rm K}.
\end{eqnarray*}

All the elements of $\Kappa$ commute with all the elements of $\Nu$. 
The group $\Gamma$ is the direct product of ${\rm K}$ and ${\rm N}$ 
if and only if ${\rm K} = \Lambda$. 
This corresponds to the occurrence of a row of $0\!+ \cdots 0\!+0-$ 
(resp. $0\!+\cdots 0+$) in the couplings column of Table 1 of \cite{C-T} 
for interval (resp. circular) fibrations.

The group $\Lambda$ is infinite cyclic 
if and only if all the elements of $\Lambda$ are translations. 
This corresponds to the occurrence of a row of 
all plus signs in the couplings column of Table 1 of \cite{C-T}. 

The tetragonal and hexagonal families of 3-dimensional space group isomorphism types  
(IT numbers 75-194) consist of all the reducible group isomorphism types 
whose point group has a unique 1-dimensional invariant vector space. 
If $\Gamma$ belongs to one of these 110 isomorphism types, 
then $\Gamma$ has a unique 1-dimensional, complete, normal subgroup $\Kappa$ and 
a unique 2-dimensional, complete, normal subgroup $\Nu$;  
moreover $\Kappa$ and $\Nu$ are orthogonal. 
This is the case if and only if in the Conway et al.~representation of $\Gamma$  
the group $\Eta$ is irreducible. 
Hence, the classification of the geometric co-Seifert fibrations of $E^3/\Gamma$ 
corresponds via orthogonal duality to the Conway et al.~\ classification of the geometric 
Seifert fibrations of $E^3/\Gamma$ when $\Gamma$ belongs to the tetragonal 
or hexagonal families. 

The orthorhombic family (IT numbers 16-74) consists of all the reducible group isomorphism types 
whose point group has exactly three orthogonal 1-dimensional invariant vector spaces. 
If $\Gamma$ belongs to one of these 59 isomorphism types, 
then $\Gamma$ has exactly three 1-dimensional, complete, normal subgroups 
$\Kappa_1,\Kappa_2,\Kappa_3$,    
and $\Gamma$ has exactly three 2-mensional, complete, normal subgroups $\Nu_1,\Nu_2,\Nu_3$.  
This is the case if and only if in the Conway et al.~\ representation of $\Gamma$ 
the group $\Eta$ has exactly two proper complete, normal subgroups. 
We order $\Kappa_i$ and $\Nu_i$ so that $V_i ={\rm Span}(\Kappa_i)$ 
and $W_i={\rm Span}(\Nu_i)$ are orthogonal complements for each $i$. 
The 1-dimensional vector spaces $V_1,V_2,V_3$ are mutually orthogonal. 

Suppose $\alpha:E^3\to E^3$ is an affine homeomorphism 
such that $\alpha\Gamma\alpha^{-1} = \Gamma'$ with $\Gamma'$ a space group. 
Define $\Kappa_i' = \alpha\Kappa_i\alpha^{-1}$ 
and $\Nu_i' = (\Kappa_i')^\perp$ for $i=1,2,3$. 
Let $V_i' ={\rm Span}(\Kappa_i')$ and $W_i' = {\rm Span}(\Nu_i')$ for each $i$. 
Let $\alpha = a + A$ with $A$ a linear automorphism of $E^3$. 
Then $AV_i = V_i'$ for each $i$. 
Now $\alpha\Nu_i\alpha^{-1} = \Nu_j'$ for some $j$ by Corollary 2. 
Therefore $AW_i = W_j'$. 
As $AV_i\cap AW_i = \{0\}$, we must have that $j=i$ for each $i=1,2,3$. 
Likewise if $\Nu_i' = \alpha\Nu_i\alpha^{-1}$ 
and $\Kappa_i' = (\Nu_i')^\perp$ for each $i=1,2,3$,  
then $\alpha(\Kappa_i)\alpha^{-1} = \Kappa_i'$ for each $i=1,2,3$. 
Thus if $\Nu$ is a complete normal subgroup of $\Gamma$,  
then $\alpha(\Nu^\perp)\alpha^{-1} = (\alpha\Nu\alpha^{-1})^\perp$. 
By Theorem 10, the geometric fibrations of $E^3/\Gamma$ and $E^3/\Gamma'$ 
determined by complete normal subgroups $\Nu$ of $\Gamma$ 
and $\Nu'$ of $\Gamma'$ are affinely equivalent 
if and only if the geometric fibrations of $E^3/\Gamma$ and $E^3/\Gamma'$ 
determined by $\Nu^\perp$ and $(\Nu')^\perp$ are affinely equivalent.  
Hence, the classification of the geometric co-Seifert fibrations of $E^3/\Gamma$ 
corresponds via orthogonal duality to the Conway et al.~\ classification of the geometric 
Seifert fibrations of $E^3/\Gamma$ when $\Gamma$ belongs to the orthorhombic family. 

The triclinic and monoclinic families consists of all the reducible space group isomorphism types 
whose point group has infinitely many 1-dimensional invariant vector spaces. 
If $\Gamma$ belongs to one of these 15 isomorphism types, 
then $\Gamma$ has infinitely many 1-dimensional, complete, normal subgroups,  
and $\Gamma$ has infinitely many 2-dimensional, complete, normal subgroups. 
This is the case if and only if in the Conway et al. representation of $\Gamma$ 
the group $\Eta$ has infinitely many proper complete normal subgroups. 

The triclinic family consists of two isomorphism types (IT numbers 1 and 2). 
If $\Gamma$ belongs to the triclinic family, 
then all the geometric Seifert fibrations of $E^3/\Gamma$ 
are affinely equivalent and all the geometric co-Seifert fibrations of $E^3/\Gamma$ 
are affinely equivalent. 

The monoclinic family consists of 13 isomorphism types (IT numbers 3-15).  
If $\Gamma$ belongs to one of these 13 isomorphism types, 
then $\Gamma$ has a unique 1-dimensional, complete, characteristic subgroup $\Kappa$  
and a unique 2-dimensional, complete, characteristic subgroup $\Nu$; 
moreover $\Kappa$ and $\Nu$ are orthogonal.  
The corresponding Seifert fibration is given by the primary name in Table 2b 
of \cite{C-T}. 
The classification of the geometric co-Seifert fibrations of $E^3/\Gamma$ 
corresponds via orthogonal duality to the Conway et al.~\ classification of the geometric 
Seifert fibrations of $E^3/\Gamma$ when $\Gamma$ belongs to the monoclinic family 
except for the two cases $(\ast\!:\!\times)$ and $(2\ov{\ast}2\!:\!2)$ in \cite{C-T}.  
In Table 1, we replace the Seifert fibration $(\ast\!:\!\times)$ 
of space group IT 9 
with the affine equivalent Seifert fibration $(\ast\!:\!\times_1)$,  
with couplings $\frac{1}{2}\!+\frac{1}{2}+$, 
and we replace the Seifert fibration $(2\ov{\ast}2\!:\!2)$ of space group IT 15
with the affine equivalent Seifert fibration $(2\ov{\ast}2_1\!:\!2)$ 
with couplings $0\!-\frac{1}{2}\!-\frac{1}{2}+$. 
With these two substitutions, indicated by a $\dag$ in Table 1, 
the classification of the geometric co-Seifert 
fibrations of $E^3/\Gamma$ under affine equivalence corresponds via orthogonal duality  
to the Conway et al.~\ classification of the geometric Seifert fibrations of $E^3/\Gamma$. 
See Table 1 for the orthogonal correspondence between the the classifications 
of the Seifert and co-Seifert fibrations of $E^3/\Gamma$.  

The first column of Table 1 is the IT number of the corresponding space group. 
The second column is the Conway fibrifold name of the Seifert fibration. 
The third column indicates whether or not the corresponding space group extension splits. 
If the generic fiber is a closed interval, then the space group extension splits 
by Theorem 18. 
The fourth column indicates the generic fiber of the orthogonal dual co-Seifert fibration. 
The fifth column indicates the base of the orthogonal dual co-Seifert fibration 
with a centered dot representing a circle and a dash representing a closed interval.  
The sixth column indicates whether or not the corresponding space group extension splits. 
If the base is a circle, then the space group extension splits, 
since $\Gamma/\Nu$ is an infinite cyclic group. 
The seventh column gives the index of $\Kappa\Nu$ in $\Gamma$; 
in particular, the index is 1 if and only if both fibrations are direct products. 
If the generic fiber of the Seifert fibration is a closed interval, 
then the base of the co-Seifert fibration is also a closed interval 
and the index is 1 or 2 by Theorem 20. 
The information in Table 1 was obtained by computer calculations. 
Finally, the 10 closed flat space forms in Table 1 
have IT numbers 1,4,7,9,19,29,33,76,144,169.

\newpage

\begingroup
\setlength{\columnsep}{-9pt}
\footnotesize
\setlength{\tabcolsep}{1.6pt}
\newcommand{\fibtablestart}{
\begin{tabular}{clclccc}
\multicolumn{3}{c}{\underline{Seifert fibration}}&
\multicolumn{4}{c}{\underline{co-Seifert fibration}}\\
no.&fibrifold&split& fiber&base&split&ind.\\
\cline{1-7}}
\newcommand{\fibtableend}{
\end{tabular}\vfill}
\newcommand{\fibtablesplit}{
\fibtableend
\fibtablestart}
\newcommand{\fibtablepage}{
\fibtableend
\end{multicols}
\begin{multicols}{2}
\fibtablestart}
\begin{multicols}{2}
\fibtablestart
1&$({\circ})$&Yes&${\circ}$&$\cdot$&Yes&1\\
2&$(2222)$&Yes&${\circ}$&$-$&Yes&2\\
3&$({\ast}_0{\ast}_0)$&Yes&${\circ}$&$-$&Yes&2\\
3&$(2_02_02_02_0)$&Yes&$2222$&$\cdot$&Yes&1\\
4&$({\bar\times\,}{\bar\times\,})$&Yes&${\circ}$&$-$&No&2\\
4&$(2_12_12_12_1)$&No&${\circ}$&$\cdot$&Yes&2\\
5&$({\ast}{\bar\times\,})$&Yes&${\circ}$&$-$&Yes&2\\
5&$({\ast}_1{\ast}_1)$&No&${\circ}$&$-$&No&4\\
5&$(2_02_02_12_1)$&No&$2222$&$\cdot$&Yes&2\\
6&$[{\circ}_0]$&Yes&${\circ}$&$-$&Yes&1\\
6&$({\ast}{\cdot}{\ast}{\cdot})$&Yes&${\ast}{\ast}$&$\cdot$&Yes&1\\
7&$({\bar\circ}_0)$&Yes&${\circ}$&$-$&No&2\\
7&$({\times}{\times}_0)$&Yes&${\times}{\times}$&$\cdot$&Yes&1\\
7&$({\ast}{:}{\ast}{:})$&No&${\circ}$&$\cdot$&Yes&2\\
8&$[{\circ}_1]$&Yes&${\circ}$&$-$&No&2\\
8&$({\ast}{\cdot}{\times})$&Yes&${\ast}{\times}$&$\cdot$&Yes&1\\
8&$({\ast}{\cdot}{\ast}{:})$&No&${\ast}{\ast}$&$\cdot$&Yes&2\\
9&$({\bar\circ}_1)$&No&${\circ}$&$-$&No&4\\
9&$({\times}{\times}_1)$&No&${\times}{\times}$&$\cdot$&Yes&2\\
9&$({\ast}{:}{\times}_1)^{\dag}$&No&${\circ}$&$\cdot$&Yes&2\\
10&$[2_02_02_02_0]$&Yes&$2222$&$-$&Yes&1\\
10&$({\ast}2{\cdot}22{\cdot}2)$&Yes&${\ast}{\ast}$&$-$&Yes&2\\
11&$[2_12_12_12_1]$&Yes&${\circ}$&$-$&Yes&2\\
11&$(22{\ast}{\cdot})$&Yes&${\ast}{\ast}$&$-$&Yes&2\\
12&$[2_02_02_12_1]$&Yes&$2222$&$-$&Yes&2\\
12&$(2{\bar\ast}2{\cdot}2)$&Yes&${\ast}{\times}$&$-$&Yes&2\\
12&$({\ast}2{\cdot}22{:}2)$&No&${\ast}{\ast}$&$-$&Yes&4\\
13&$(2_02_022)$&Yes&$2222$&$-$&Yes&2\\
13&$(22{\ast}_0)$&Yes&${\times}{\times}$&$-$&Yes&2\\
13&$({\ast}2{:}22{:}2)$&No&${\circ}$&$-$&Yes&4\\
14&$(2_12_122)$&No&${\circ}$&$-$&No&4\\
14&$(22{\times})$&Yes&${\times}{\times}$&$-$&Yes&2\\
14&$(22{\ast}{:})$&No&${\circ}$&$-$&No&4\\
15&$(2_02_122)$&No&$2222$&$-$&Yes&4\\
15&$(22{\ast}_1)$&No&${\times}{\times}$&$-$&Yes&4\\
15&$(2{\bar\ast}2_1{:}2)^{\dag}$&No&${\circ}$&$-$&Yes&4\\
16&$({\ast}2_02_02_02_0)$&Yes&$2222$&$-$&Yes&2\\
17&$(2_02_0{\ast})$&Yes&$2222$&$-$&Yes&2\\
17&$({\ast}2_12_12_12_1)$&No&${\circ}$&$-$&Yes&4\\
18&$(2_02_0{\bar\times\,})$&Yes&$2222$&$-$&No&2\\
18&$(2_12_1{\ast})$&No&${\circ}$&$-$&No&4\\
19&$(2_12_1{\bar\times\,})$&No&${\circ}$&$-$&No&4\\
20&$(2_02_1{\ast})$&No&$2222$&$-$&No&4\\
20&$(2_1{\ast}2_12_1)$&No&${\circ}$&$-$&Yes&4\\
21&$(2_0{\ast}2_02_0)$&Yes&$2222$&$-$&Yes&2\\
21&$({\ast}2_02_02_12_1)$&No&$2222$&$-$&Yes&4\\
22&$({\ast}2_02_12_02_1)$&No&$2222$&$-$&Yes&4\\
23&$(2_1{\ast}2_02_0)$&No&$2222$&$-$&No&4\\
24&$(2_0{\ast}2_12_1)$&No&$2222$&$-$&Yes&4\\
25&$[{\ast}_0{\cdot}{\ast}_0{\cdot}]$&Yes&${\ast}{\ast}$&$-$&Yes&1\\
25&$({\ast}{\cdot}2{\cdot}2{\cdot}2{\cdot}2)$&Yes&${\ast}2222$&$\cdot$&Yes&1\\
26&$[{\times}_0{\times}_0]$&Yes&${\times}{\times}$&$-$&Yes&1\\
\fibtablesplit
26&$({\bar\ast}{\cdot}{\bar\ast}{\cdot})$&Yes&${\ast}{\ast}$&$-$&No&2\\
26&$({\ast}{\cdot}2{:}2{\cdot}2{:}2)$&No&${\ast}{\ast}$&$\cdot$&Yes&2\\
27&$({\bar\ast}_0{\bar\ast}_0)$&Yes&${\times}{\times}$&$-$&Yes&2\\
27&$({\ast}{:}2{:}2{:}2{:}2)$&No&$2222$&$\cdot$&Yes&2\\
28&$({\ast}{\cdot}{\ast}_0)$&Yes&${\ast}{\ast}$&$-$&Yes&2\\
28&$[{\ast}_0{:}{\ast}_0{:}]$&Yes&${\circ}$&$-$&Yes&2\\
28&$(2_02_0{\ast}{\cdot})$&Yes&$22{\ast}$&$\cdot$&Yes&1\\
29&$({\bar\times\,}{\times}_0)$&Yes&${\times}{\times}$&$-$&No&2\\
29&$({\bar\ast}{:}{\bar\ast}{:})$&No&${\circ}$&$-$&No&4\\
29&$(2_12_1{\ast}{:})$&No&${\times}{\times}$&$\cdot$&Yes&2\\
30&$({\ast}_0{\times}_0)$&Yes&${\times}{\times}$&$-$&Yes&2\\
30&$({\bar\ast}_1{\bar\ast}_1)$&No&${\circ}$&$-$&No&4\\
30&$(2_02_0{\ast}{:})$&No&$2222$&$\cdot$&Yes&2\\
31&$[{\times}_1{\times}_1]$&Yes&${\circ}$&$-$&No&2\\
31&$({\ast}{\cdot}{\bar\times\,})$&Yes&${\ast}{\ast}$&$-$&No&2\\
31&$(2_12_1{\ast}{\cdot})$&No&${\ast}{\ast}$&$\cdot$&Yes&2\\
32&$({\ast}{:}{\ast}_0)$&No&${\circ}$&$-$&No&4\\
32&$(2_02_0{\times}_0)$&Yes&$22{\times}$&$\cdot$&Yes&1\\
33&$({\bar\times\,}{\times}_1)$&No&${\circ}$&$-$&No&4\\
33&$({\ast}{:}{\bar\times\,})$&No&${\circ}$&$-$&No&4\\
33&$(2_12_1{\times})$&No&${\times}{\times}$&$\cdot$&Yes&2\\
34&$({\ast}_0{\times}_1)$&No&${\circ}$&$-$&No&4\\
34&$(2_02_0{\times}_1)$&No&$2222$&$\cdot$&Yes&2\\
35&$[{\ast}_0{\cdot}{\ast}_0{:}]$&Yes&${\ast}{\ast}$&$-$&Yes&2\\
35&$(2_0{\ast}{\cdot}2{\cdot}2)$&Yes&$2{\ast}22$&$\cdot$&Yes&1\\
36&$[{\times}_0{\times}_1]$&Yes&${\times}{\times}$&$-$&No&2\\
36&$({\bar\ast}{\cdot}{\bar\ast}{:})$&No&${\ast}{\ast}$&$-$&No&4\\
36&$(2_1{\ast}{\cdot}2{:}2)$&No&${\ast}{\times}$&$\cdot$&Yes&2\\
37&$({\bar\ast}_0{\bar\ast}_1)$&No&${\times}{\times}$&$-$&Yes&4\\
37&$(2_0{\ast}{:}2{:}2)$&No&$2222$&$\cdot$&Yes&2\\
38&$[{\ast}{\cdot}{\times}_0]$&Yes&${\ast}{\times}$&$-$&Yes&1\\
38&$[{\ast}_1{\cdot}{\ast}_1{\cdot}]$&Yes&${\ast}{\ast}$&$-$&No&2\\
38&$({\ast}{\cdot}2{\cdot}2{\cdot}2{:}2)$&No&${\ast}2222$&$\cdot$&Yes&2\\
39&$({\bar\ast}{\cdot}{\bar\ast}_0)$&Yes&${\ast}{\times}$&$-$&Yes&2\\
39&$[{\ast}_1{:}{\ast}_1{:}]$&Yes&${\times}{\times}$&$-$&Yes&2\\
39&$({\ast}{\cdot}2{:}2{:}2{:}2)$&No&$22{\ast}$&$\cdot$&Yes&2\\
40&$[{\ast}{:}{\times}_1]$&Yes&${\circ}$&$-$&Yes&2\\
40&$({\ast}{\cdot}{\ast}_1)$&No&${\ast}{\ast}$&$-$&No&4\\
40&$(2_02_1{\ast}{\cdot})$&No&$22{\ast}$&$\cdot$&Yes&2\\
41&$({\bar\ast}{:}{\bar\ast}_1)$&No&${\circ}$&$-$&No&4\\
41&$({\ast}{:}{\ast}_1)$&No&${\times}{\times}$&$-$&No&4\\
41&$(2_02_1{\ast}{:})$&No&$22{\times}$&$\cdot$&Yes&2\\
42&$[{\ast}_1{\cdot}{\ast}_1{:}]$&Yes&${\ast}{\times}$&$-$&Yes&2\\
42&$({\ast}{\cdot}2{\cdot}2{:}2{:}2)$&No&$2{\ast}22$&$\cdot$&Yes&2\\
43&$({\ast}_1{\times})$&No&${\circ}$&$-$&No&8\\
43&$(2_02_1{\times})$&No&$2222$&$\cdot$&Yes&4\\
44&$[{\ast}{\cdot}{\times}_1]$&Yes&${\ast}{\ast}$&$-$&No&2\\
44&$(2_1{\ast}{\cdot}2{\cdot}2)$&No&${\ast}2222$&$\cdot$&Yes&2\\
45&$({\bar\ast}{:}{\bar\ast}_0)$&No&${\times}{\times}$&$-$&No&4\\
45&$(2_1{\ast}{:}2{:}2)$&No&$22{\times}$&$\cdot$&Yes&2\\
46&$[{\ast}{:}{\times}_0]$&Yes&${\times}{\times}$&$-$&Yes&2\\
46&$({\bar\ast}{\cdot}{\bar\ast}_1)$&No&${\ast}{\ast}$&$-$&No&4\\
\fibtableend\end{multicols}
\medskip\centerline{Table 1. Three-dimensional crystallographic Seifert and co-Seifert fibrations.}\goodbreak
\begin{multicols}{2}\fibtablestart
46&$(2_0{\ast}{\cdot}2{:}2)$&No&$22{\ast}$&$\cdot$&Yes&2\\
47&$[{\ast}{\cdot}2{\cdot}2{\cdot}2{\cdot}2]$&Yes&${\ast}2222$&$-$&Yes&1\\
48&$(2{\bar\ast}_12_02_0)$&No&$2222$&$-$&Yes&4\\
49&$({\ast}2_02_02{\cdot}2)$&Yes&$22{\ast}$&$-$&Yes&2\\
49&$[{\ast}{:}2{:}2{:}2{:}2]$&Yes&$2222$&$-$&Yes&2\\
50&$(2{\bar\ast}_02_02_0)$&Yes&$22{\times}$&$-$&Yes&2\\
50&$({\ast}2_02_02{:}2)$&No&$2222$&$-$&Yes&4\\
51&$[2_02_0{\ast}{\cdot}]$&Yes&$22{\ast}$&$-$&Yes&1\\
51&$({\ast}2{\cdot}2{\cdot}2{\cdot}2)$&Yes&${\ast}2222$&$-$&Yes&2\\
51&$[{\ast}{\cdot}2{:}2{\cdot}2{:}2]$&Yes&${\ast}{\ast}$&$-$&Yes&2\\
52&$(2_02{\bar\ast}_1)$&No&$2222$&$-$&Yes&4\\
52&$(2{\bar\ast}2_12_1)$&No&${\times}{\times}$&$-$&Yes&4\\
52&$(2_0{\ast}2{:}2)$&No&$2222$&$-$&Yes&4\\
53&$[2_02_0{\ast}{:}]$&Yes&$2222$&$-$&Yes&2\\
53&$(2_0{\ast}2{\cdot}2)$&Yes&$22{\ast}$&$-$&Yes&2\\
53&$({\ast}2_12_12{\cdot}2)$&No&${\ast}{\ast}$&$-$&Yes&4\\
54&$(2_02{\bar\ast}_0)$&Yes&$22{\times}$&$-$&Yes&2\\
54&$({\ast}2_12_12{:}2)$&No&${\times}{\times}$&$-$&Yes&4\\
54&$({\ast}2{:}2{:}2{:}2)$&No&$2222$&$-$&Yes&4\\
55&$[2_02_0{\times}_0]$&Yes&$22{\times}$&$-$&Yes&1\\
55&$({\ast}2{\cdot}2{:}2{\cdot}2)$&No&${\ast}{\ast}$&$-$&No&4\\
56&$(2_12{\bar\ast}_0)$&No&${\times}{\times}$&$-$&Yes&4\\
56&$(2{\bar\ast}{:}2{:}2)$&No&$2222$&$-$&No&4\\
57&$(2_02{\bar\ast}{\cdot})$&Yes&$22{\ast}$&$-$&Yes&2\\
57&$[2_12_1{\ast}{:}]$&Yes&${\times}{\times}$&$-$&Yes&2\\
57&$({\ast}2{:}2{\cdot}2{:}2)$&No&${\ast}{\ast}$&$-$&Yes&4\\
58&$[2_02_0{\times}_1]$&Yes&$2222$&$-$&No&2\\
58&$(2_1{\ast}2{\cdot}2)$&No&${\ast}{\ast}$&$-$&No&4\\
59&$[2_12_1{\ast}{\cdot}]$&Yes&${\ast}{\ast}$&$-$&Yes&2\\
59&$(2{\bar\ast}{\cdot}2{\cdot}2)$&Yes&${\ast}2222$&$-$&Yes&2\\
60&$(2_12{\bar\ast}_1)$&No&${\times}{\times}$&$-$&No&4\\
60&$(2_02{\bar\ast}{:})$&No&$2222$&$-$&No&4\\
60&$(2_1{\ast}2{:}2)$&No&${\times}{\times}$&$-$&Yes&4\\
61&$(2_12{\bar\ast}{:})$&No&${\times}{\times}$&$-$&No&4\\
62&$[2_12_1{\times}]$&Yes&${\times}{\times}$&$-$&Yes&2\\
62&$(2_12{\bar\ast}{\cdot})$&No&${\ast}{\ast}$&$-$&No&4\\
62&$(2{\bar\ast}{\cdot}2{:}2)$&No&${\ast}{\ast}$&$-$&No&4\\
63&$[2_02_1{\ast}{\cdot}]$&Yes&$22{\ast}$&$-$&Yes&2\\
63&$[2_1{\ast}{\cdot}2{:}2]$&Yes&${\ast}{\times}$&$-$&Yes&2\\
63&$({\ast}2{\cdot}2{\cdot}2{:}2)$&No&${\ast}2222$&$-$&Yes&4\\
64&$[2_02_1{\ast}{:}]$&Yes&$22{\times}$&$-$&Yes&2\\
64&$({\ast}2_12{\cdot}2{:}2)$&No&${\ast}{\times}$&$-$&Yes&4\\
64&$({\ast}2{\cdot}2{:}2{:}2)$&No&$22{\ast}$&$-$&Yes&4\\
65&$[2_0{\ast}{\cdot}2{\cdot}2]$&Yes&$2{\ast}22$&$-$&Yes&1\\
65&$[{\ast}{\cdot}2{\cdot}2{\cdot}2{:}2]$&Yes&${\ast}2222$&$-$&Yes&2\\
66&$[2_0{\ast}{:}2{:}2]$&Yes&$2222$&$-$&Yes&2\\
66&$({\ast}2_02_12{\cdot}2)$&No&$22{\ast}$&$-$&Yes&4\\
67&$({\ast}2_02{\cdot}2{\cdot}2)$&Yes&$2{\ast}22$&$-$&Yes&2\\
67&$[{\ast}{\cdot}2{:}2{:}2{:}2]$&Yes&$22{\ast}$&$-$&Yes&2\\
68&$({\ast}2_02_12{:}2)$&No&$22{\times}$&$-$&Yes&4\\
68&$({\ast}2_02{:}2{:}2)$&No&$2222$&$-$&Yes&4\\
69&$[{\ast}{\cdot}2{\cdot}2{:}2{:}2]$&Yes&$2{\ast}22$&$-$&Yes&2\\
\fibtablesplit
70&$(2{\bar\ast}2_02_1)$&No&$2222$&$-$&Yes&8\\
71&$[2_1{\ast}{\cdot}2{\cdot}2]$&Yes&${\ast}2222$&$-$&Yes&2\\
72&$[2_1{\ast}{:}2{:}2]$&Yes&$22{\times}$&$-$&Yes&2\\
72&$({\ast}2_02{\cdot}2{:}2)$&No&$22{\ast}$&$-$&Yes&4\\
73&$({\ast}2_12{:}2{:}2)$&No&$22{\times}$&$-$&Yes&4\\
74&$[2_0{\ast}{\cdot}2{:}2]$&Yes&$22{\ast}$&$-$&Yes&2\\
74&$({\ast}2_12{\cdot}2{\cdot}2)$&No&${\ast}2222$&$-$&Yes&4\\
75&$(4_04_02_0)$&Yes&$442$&$\cdot$&Yes&1\\
76&$(4_14_12_1)$&No&${\circ}$&$\cdot$&Yes&4\\
77&$(4_24_22_0)$&No&$2222$&$\cdot$&Yes&2\\
79&$(4_24_02_1)$&No&$442$&$\cdot$&Yes&2\\
80&$(4_34_12_0)$&No&$2222$&$\cdot$&Yes&4\\
81&$(442_0)$&Yes&$2222$&$-$&No&2\\
82&$(442_1)$&No&$2222$&$-$&No&4\\
83&$[4_04_02_0]$&Yes&$442$&$-$&Yes&1\\
84&$[4_24_22_0]$&Yes&$2222$&$-$&No&2\\
85&$(44_02)$&Yes&$442$&$-$&Yes&2\\
86&$(44_22)$&No&$2222$&$-$&No&4\\
87&$[4_24_02_1]$&Yes&$442$&$-$&Yes&2\\
88&$(44_12)$&No&$2222$&$-$&No&8\\
89&$({\ast}4_04_02_0)$&Yes&$442$&$-$&Yes&2\\
90&$(4_0{\ast}2_0)$&Yes&$442$&$-$&Yes&2\\
91&$({\ast}4_14_12_1)$&No&${\circ}$&$-$&Yes&8\\
92&$(4_1{\ast}2_1)$&No&${\circ}$&$-$&No&8\\
93&$({\ast}4_24_22_0)$&No&$2222$&$-$&Yes&4\\
94&$(4_2{\ast}2_0)$&No&$2222$&$-$&No&4\\
97&$({\ast}4_24_02_1)$&No&$442$&$-$&Yes&4\\
98&$({\ast}4_34_12_0)$&No&$2222$&$-$&Yes&8\\
99&$({\ast}{\cdot}4{\cdot}4{\cdot}2)$&Yes&${\ast}442$&$\cdot$&Yes&1\\
100&$(4_0{\ast}{\cdot}2)$&Yes&$4{\ast}2$&$\cdot$&Yes&1\\
101&$({\ast}{:}4{\cdot}4{:}2)$&No&$2{\ast}22$&$\cdot$&Yes&2\\
102&$(4_2{\ast}{\cdot}2)$&No&$2{\ast}22$&$\cdot$&Yes&2\\
103&$({\ast}{:}4{:}4{:}2)$&No&$442$&$\cdot$&Yes&2\\
104&$(4_0{\ast}{:}2)$&No&$442$&$\cdot$&Yes&2\\
105&$({\ast}{\cdot}4{:}4{\cdot}2)$&No&${\ast}2222$&$\cdot$&Yes&2\\
106&$(4_2{\ast}{:}2)$&No&$22{\times}$&$\cdot$&Yes&2\\
107&$({\ast}{\cdot}4{\cdot}4{:}2)$&No&${\ast}442$&$\cdot$&Yes&2\\
108&$({\ast}{\cdot}4{:}4{:}2)$&No&$4{\ast}2$&$\cdot$&Yes&2\\
109&$(4_1{\ast}{\cdot}2)$&No&${\ast}2222$&$\cdot$&Yes&4\\
110&$(4_1{\ast}{:}2)$&No&$22{\times}$&$\cdot$&Yes&4\\
111&$({\ast}4{\cdot}42_0)$&Yes&$2{\ast}22$&$-$&Yes&2\\
112&$({\ast}4{:}42_0)$&No&$2222$&$-$&No&4\\
113&$(4{\bar\ast}{\cdot}2)$&Yes&$2{\ast}22$&$-$&No&2\\
114&$(4{\bar\ast}{:}2)$&No&$2222$&$-$&No&4\\
115&$({\ast}{\cdot}44{\cdot}2)$&Yes&${\ast}2222$&$-$&Yes&2\\
116&$({\ast}{:}44{:}2)$&No&$2222$&$-$&No&4\\
117&$(4{\bar\ast}_02_0)$&Yes&$22{\times}$&$-$&Yes&2\\
118&$(4{\bar\ast}_12_0)$&No&$2222$&$-$&No&4\\
119&$({\ast}4{\cdot}42_1)$&No&${\ast}2222$&$-$&Yes&4\\
120&$({\ast}4{:}42_1)$&No&$22{\times}$&$-$&Yes&4\\
121&$({\ast}{\cdot}44{:}2)$&No&$2{\ast}22$&$-$&No&4\\
122&$(4{\bar\ast}2_1)$&No&$2222$&$-$&No&8\\
\fibtableend\end{multicols}
\medskip\centerline{Table 1. (cont.) Three-dimensional crystallographic Seifert and co-Seifert fibrations.}\goodbreak
\begin{multicols}{2}\fibtablestart
123&$[{\ast}{\cdot}4{\cdot}4{\cdot}2]$&Yes&${\ast}442$&$-$&Yes&1\\
124&$[{\ast}{:}4{:}4{:}2]$&Yes&$442$&$-$&Yes&2\\
125&$({\ast}4_04{\cdot}2)$&Yes&$4{\ast}2$&$-$&Yes&2\\
126&$({\ast}4_04{:}2)$&No&$442$&$-$&Yes&4\\
127&$[4_0{\ast}{\cdot}2]$&Yes&$4{\ast}2$&$-$&Yes&1\\
128&$[4_0{\ast}{:}2]$&Yes&$442$&$-$&Yes&2\\
129&$({\ast}4{\cdot}4{\cdot}2)$&Yes&${\ast}442$&$-$&Yes&2\\
130&$({\ast}4{:}4{:}2)$&No&$442$&$-$&Yes&4\\
131&$[{\ast}{\cdot}4{:}4{\cdot}2]$&Yes&${\ast}2222$&$-$&Yes&2\\
132&$[{\ast}{:}4{\cdot}4{:}2]$&Yes&$2{\ast}22$&$-$&Yes&2\\
133&$({\ast}4_24{:}2)$&No&$22{\times}$&$-$&Yes&4\\
134&$({\ast}4_24{\cdot}2)$&No&$2{\ast}22$&$-$&Yes&4\\
135&$[4_2{\ast}{:}2]$&Yes&$22{\times}$&$-$&Yes&2\\
136&$[4_2{\ast}{\cdot}2]$&Yes&$2{\ast}22$&$-$&No&2\\
137&$({\ast}4{\cdot}4{:}2)$&No&${\ast}2222$&$-$&Yes&4\\
138&$({\ast}4{:}4{\cdot}2)$&No&$2{\ast}22$&$-$&No&4\\
139&$[{\ast}{\cdot}4{\cdot}4{:}2]$&Yes&${\ast}442$&$-$&Yes&2\\
140&$[{\ast}{\cdot}4{:}4{:}2]$&Yes&$4{\ast}2$&$-$&Yes&2\\
141&$({\ast}4_14{\cdot}2)$&No&${\ast}2222$&$-$&Yes&8\\
142&$({\ast}4_14{:}2)$&No&$22{\times}$&$-$&Yes&8\\
143&$(3_03_03_0)$&Yes&$333$&$\cdot$&Yes&1\\
144&$(3_13_13_1)$&No&${\circ}$&$\cdot$&Yes&3\\
146&$(3_03_13_2)$&No&$333$&$\cdot$&Yes&3\\
147&$(63_02)$&Yes&$333$&$-$&Yes&2\\
148&$(63_12)$&No&$333$&$-$&Yes&6\\
149&$({\ast}3_03_03_0)$&Yes&$333$&$-$&Yes&2\\
150&$(3_0{\ast}3_0)$&Yes&$333$&$-$&Yes&2\\
151&$({\ast}3_13_13_1)$&No&${\circ}$&$-$&Yes&6\\
152&$(3_1{\ast}3_1)$&No&${\circ}$&$-$&Yes&6\\
155&$({\ast}3_03_13_2)$&No&$333$&$-$&Yes&6\\
156&$({\ast}{\cdot}3{\cdot}3{\cdot}3)$&Yes&${\ast}333$&$\cdot$&Yes&1\\
157&$(3_0{\ast}{\cdot}3)$&Yes&$3{\ast}3$&$\cdot$&Yes&1\\
158&$({\ast}{:}3{:}3{:}3)$&No&$333$&$\cdot$&Yes&2\\
\fibtablesplit
159&$(3_0{\ast}{:}3)$&No&$333$&$\cdot$&Yes&2\\
160&$(3_1{\ast}{\cdot}3)$&No&${\ast}333$&$\cdot$&Yes&3\\
161&$(3_1{\ast}{:}3)$&No&$333$&$\cdot$&Yes&6\\
162&$({\ast}{\cdot}63_02)$&Yes&$3{\ast}3$&$-$&Yes&2\\
163&$({\ast}{:}63_02)$&No&$333$&$-$&Yes&4\\
164&$({\ast}6{\cdot}3{\cdot}2)$&Yes&${\ast}333$&$-$&Yes&2\\
165&$({\ast}6{:}3{:}2)$&No&$333$&$-$&Yes&4\\
166&$({\ast}{\cdot}63_12)$&No&${\ast}333$&$-$&Yes&6\\
167&$({\ast}{:}63_12)$&No&$333$&$-$&Yes&12\\
168&$(6_03_02_0)$&Yes&$632$&$\cdot$&Yes&1\\
169&$(6_13_12_1)$&No&${\circ}$&$\cdot$&Yes&6\\
171&$(6_23_22_0)$&No&$2222$&$\cdot$&Yes&3\\
173&$(6_33_02_1)$&No&$333$&$\cdot$&Yes&2\\
174&$[3_03_03_0]$&Yes&$333$&$-$&Yes&1\\
175&$[6_03_02_0]$&Yes&$632$&$-$&Yes&1\\
176&$[6_33_02_1]$&Yes&$333$&$-$&Yes&2\\
177&$({\ast}6_03_02_0)$&Yes&$632$&$-$&Yes&2\\
178&$({\ast}6_13_12_1)$&No&${\circ}$&$-$&Yes&12\\
180&$({\ast}6_23_22_0)$&No&$2222$&$-$&Yes&6\\
182&$({\ast}6_33_02_1)$&No&$333$&$-$&Yes&4\\
183&$({\ast}{\cdot}6{\cdot}3{\cdot}2)$&Yes&${\ast}632$&$\cdot$&Yes&1\\
184&$({\ast}{:}6{:}3{:}2)$&No&$632$&$\cdot$&Yes&2\\
185&$({\ast}{\cdot}6{:}3{:}2)$&No&$3{\ast}3$&$\cdot$&Yes&2\\
186&$({\ast}{:}6{\cdot}3{\cdot}2)$&No&${\ast}333$&$\cdot$&Yes&2\\
187&$[{\ast}{\cdot}3{\cdot}3{\cdot}3]$&Yes&${\ast}333$&$-$&Yes&1\\
188&$[{\ast}{:}3{:}3{:}3]$&Yes&$333$&$-$&Yes&2\\
189&$[3_0{\ast}{\cdot}3]$&Yes&$3{\ast}3$&$-$&Yes&1\\
190&$[3_0{\ast}{:}3]$&Yes&$333$&$-$&Yes&2\\
191&$[{\ast}{\cdot}6{\cdot}3{\cdot}2]$&Yes&${\ast}632$&$-$&Yes&1\\
192&$[{\ast}{:}6{:}3{:}2]$&Yes&$632$&$-$&Yes&2\\
193&$[{\ast}{\cdot}6{:}3{:}2]$&Yes&$3{\ast}3$&$-$&Yes&2\\
194&$[{\ast}{:}6{\cdot}3{\cdot}2]$&Yes&${\ast}333$&$-$&Yes&2\\
\fibtableend\vfill
\end{multicols}
\medskip\centerline{Table 1. (cont.) Three-dimensional crystallographic Seifert and co-Seifert fibrations.}
\endgroup

\end{document}